\documentclass[reqno]{amsart}
\usepackage{amsmath,amssymb,latexsym,soul,cite,mathrsfs}
\usepackage{color,enumitem,graphicx}
\usepackage[colorlinks=true,urlcolor=blue,
citecolor=red,linkcolor=blue,linktocpage,pdfpagelabels,
bookmarksnumbered,bookmarksopen]{hyperref}
\usepackage[dvipsnames,svgnames]{xcolor}
\usepackage[english]{babel}
\usepackage[left=2.8cm,right=2.8cm,top=2.9cm,bottom=2.9cm]{geometry}
\usepackage[hyperpageref]{backref}
\usepackage{graphicx,xcolor,subcaption,lipsum}
\usepackage{tikz}
\usepackage{calc}
\newcommand{\quotes}[1]{``#1''}

\usepackage{amsthm}
\usepackage{mathtools}
\DeclarePairedDelimiter{\norm}{\lVert}{\rVert}
\DeclarePairedDelimiter{\sprod}{\langle}{\rangle}

\newcommand{\de}{\mathrm{d}}

\usepackage{bbm}

\newtheorem{theorem}{Theorem}
\newtheorem{lemma}[theorem]{Lemma}

\newtheorem{proposition}[theorem]{Proposition}

\newtheorem{theoremletter}{Theorem}

\theoremstyle{definition}\newtheorem{definition}[theorem]{Definition}
\theoremstyle{remark}\newtheorem{remark}[theorem]{Remark}
\theoremstyle{definition}\newtheorem{example}[theorem]{Example}

\newtheorem*{innerthm}{\innerthmname}
\newcommand{\innerthmname}{}
\newenvironment{statement}[1]
{\renewcommand{\innerthmname}{#1}\innerthm}
{\endinnerthm}
\theoremstyle{definition}

\makeatletter
\def\namedlabel#1#2{\begingroup
	#2
	\def\@currentlabel{#2}
	\phantomsection\label{#1}\endgroup
}
\makeatother

\newcommand{\inj}{\mathrm{inj}}

\newcommand{\ud}{\;\mathrm{d}}

\DeclareMathOperator{\supp}{supp}
\DeclareMathOperator{\dist}{dist}

\DeclareMathOperator{\cat}{cat}

\DeclareMathOperator{\diam}{diam}
\DeclareMathOperator{\vol}{vol}

\title[From bubbles to clusters: Multiple solutions to the Allen--Cahn system]{From bubbles to clusters: Multiple solutions to the Allen--Cahn system} 
\thanks{This work was partially supported by Funda\c c\~ao de Amparo \`a Pesquisa do Estado de S\~ao Paulo (FAPESP), Conselho Nacional de Desenvolvimento Cient\'ifico e Tecnol\'ogico (CNPq), Fulbright Commission in Brazil, and Italian National Institute of High Mathematics (GNAMPA INdAM). 
	J.H.A. was supported by FAPESP \#2023/15139-0 and CNPq 409764/2023-0, 443594/2023-6, and \#441922/2023-6. 
	D.C. was supported by INdAM as``titolare di una borsa per l’estero dell’Istituto Nazionale di Alta Matematica'' and  GNAMPA INdAM \#CUP-E53C22001930001 and FAPESP \#2022/16097-2.
	S.N. was supported by FAPESP \#2021/05256-0, and
 FAPESP \#2023/08246-0, and CNPq \#12327/2021-8 and \#441922/2023-6
	P.P. was supported by FAPESP \#2022/16097-2, and
 FAPESP \#2023/08246-0, and CNPq \#313773/2021-1 and \#441922/2023-6.
R.P. was supported by FAPESP \#2023/07697-9.}

\author[J.H. Andrade]{Jo\~{a}o Henrique Andrade}
\author[D. Corona]{Dario Corona}
\author[S. Nardulli]{Stefano Nardulli}
\author[P. Piccione]{Paolo Piccione}
\author[R. Ponciano]{Raon\'i Ponciano}

\address[J.H. Andrade]{Institute of Mathematics and Statistics,
	University of S\~ao Paulo
	\newline\indent 
05508-090, S\~ao Paulo-SP, Brazil}
\email{\href{mailto:andradejh@ime.usp.br}{andradejh@ime.usp.br}}

\address[D. Corona]{
	School of Science and Technology,
	University of Camerino
	\newline\indent 
62032, Camerino-MC, Italy}
\email{\href{mailto:dario.corona@unicam.it}{dario.corona@unicam.it}}

\address[S. Nardulli]{Department of Mathematics,
	Federal University of ABC
	\newline\indent 
09210-580, S\~ao Paulo-SP, Brazil}
\email{\href{mailto:stefano.nardulli@ufabc.edu.br}{stefano.nardulli@ufabc.edu.br}}

\address[P. Piccione]{
	Department of Mathematics, 
	School of Sciences, Great Bay University
	\newline\indent 
	523000, Dongguan-GD, People’s Republic of China
	\newline\indent 
	and
	\newline\indent
	School of Mathematical Sciences, Zhejiang Normal University
	\newline\indent 
	321004, Jinhua-ZJ, People’s Republic of China
	\newline\indent 
	and
	\newline\indent
	(Permanent address) Institute of Mathematics and Statistics,	University of S\~ao Paulo
	\newline\indent 
05508-090, S\~ao Paulo-SP, Brazil}
\email{\href{mailto:piccione@ime.usp.br}{paolo.piccione@usp.br}}

\address[R. Ponciano]{Department of Mathematics,
	Federal University of ABC
	\newline\indent 
09210-580, S\~ao Paulo-SP, Brazil}
\email{\href{mailto:raoni.ponciano@ufabc.edu.br}{raoni.ponciano@ufabc.edu.br}}


\begin{document}

\begin{abstract}
	We extend previous works on the multiplicity of solutions to the Allen-Cahn
	system on closed Riemannian manifolds by considering an arbitrary number of
	phases.
	Specifically, we show that on parallelizable manifolds
	the number of solutions is bounded from below by topological invariants
	of the underlying manifold,
	provided the temperature parameter and volume constraint
	are sufficiently small.
	The Allen-Cahn system naturally arises in phase separation models, where
	solutions represent the distribution of distinct phases in a multi-component
	mixture.
	As the temperature parameter approaches zero,
	the system’s energy approximates the multi-isoperimetric profile,
	leading to solutions concentrating in regions
	resembling isoperimetric clusters.
	For two or three phases,
	these results rely on classifying isoperimetric clusters.
	However, this classification is incomplete for a larger number of phases.
	To address this technical issue,
	we employ a ``volume-fixing variations'' approach,
	enabling us to establish our results for any number of phases
	and small volume constraints.
	This offers more profound insights into phase separation
	phenomena on manifolds with arbitrary geometry.

\end{abstract}

\maketitle

\numberwithin{equation}{section} 
\numberwithin{theorem}{section}

\section{Introduction}

The Yau conjecture \cite{MR645762} concerns the abundance
(specifically, the existence, multiplicity, and density)
of minimal hypersurfaces (codimension-one submanifolds) on a closed Riemannian manifold.
Recently, this conjecture has been affirmatively proved \cite{MR3674223, MR3953507, MR4564260} (at least generic or for \quotes{low} dimensions). 
After that, it is natural to extend this question to include other geometric variational problems, such as isoperimetric sets that are constant mean curvature (CMC) hypersurfaces and multi-isoperimetric clusters.
In this setting, the situation is much more challenging, and only partial results are known for isoperimetric sets \cite{MR4011704,MR4649390}, and even fewer for multi-isoperimetric clusters \cite{MR2976521,MR4459029}.
Surprisingly, this theory is much more developed for the case of nonlocal isoperimetric sets;
the interested reader can find more details in~\cite{arXiv:2306.07100} and the references therein.
Although classical proofs for the existence of minimal hypersurfaces rely on delicate techniques from geometric measure theory, a more accessible, PDE-based approach has been developed in~\cite{MR3743704,MR3945835,MR4498838}. This approach depends on the relationship between perimeter minimizers under volume constraints and the diffuse interface of phase-transition models at sufficiently low temperatures (see~\cite{MR473971,MR866718,MR930124,MR2032110,MR1803974,MR2948876,MR1051228,MR2769110} for more details).

More formally,
since the energy of the Allen-Cahn system on a closed manifold $\Gamma$-converges to the isoperimetric profile as the relaxation parameter vanishes, one can analyze the diffuse interfaces that approximate isoperimetric sets by using suitable level sets of solutions to this system of PDEs;
this is driven by the Laplace--Beltrami operator and sourced by a temperature-rescaled nonnegative (coupling) potential having a finite number of nondegenerate vanishing global minima and suitable growth at infinity.

The first appearance of the AC system dates back to phase-separation theory for binary fluids (see \cite{MR855305,allen-cahn,cahn-hilliard}) and in van der Waals phase-transition theory \cite{vanderwaals}. 
Furthermore, it models the Gibbs free energy \cite{gibbs} of a multiphasic fluid mixture, weighted by surface tension between pure states \cite{MR1051228,MR866718,MR473971,MR930124}. 
For greater accuracy in the physical derivation of this system of PDEs, a non-local version of this free energy (Kac potential) should replace the gradient norm. 
This approach is relevant for the Ising process and dislocations in materials with microstructures (see \cite{MR1453735,MR1638739,MR3748585}) and for the theory of triblock copolymer (see \cite{MR1983191,MR3302114}, and the references therein).

Finding bounds for the critical points of the (weighted) Gibbs free energy is a mainstream problem due to its physical and geometric relevance.
For a mixture of two (isoperimetric sets) and three immiscible fluids (weighted multi-isoperimetric clusters) on an inhomogeneous media (compact Riemannian manifold with no boundary), such a lower bound was obtained in \cite{MR4396580,MR4644903} and \cite{MR4701348} (under additional conditions on the manifold, namely parallelizability), respectively.
Our main contribution to this manuscript is to extend these results to the general case of a mixture of finitely many fluids.

\bigskip
In the following, we provide an analytical formulation of our problem.
Let $(M,g)$ be a closed (compact and without boundary) and complete
Riemannian $N$-dimensional manifold with $N \ge 2$.
Throughout the whole paper, let $m$ be a fixed positive integer.
For any vector-valued function (or $m$-map) $u \colon M \to \mathbb{R}^m$,
we denote its vectorial gradient by $\nabla_g u =(\nabla_g u^1,\dots, \nabla_g u^m)$
and the vectorial Laplace–Beltrami operator by $\Delta_g u = (\Delta_g u^1,\dots, \Delta_g u^m)$;
hence, $\nabla_g u\colon TM  \to \mathbb{R}^m$, and $\Delta_g u\colon M \to \mathbb{R}^m$.
In other words, using the metric tensor $g$,
$\nabla_g u(p)$ can be seen as an $m$--vector on $T_pM$,
which means
\begin{equation*}
	\nabla_g u(p,w)
	= \left(g\left(\nabla_g u^1(p),w\right),\dots
	,g\left(\nabla_g u^m(p),w\right)\right) \in \mathbb{R}^m,
	\qquad \forall (p,w) \in TM.
\end{equation*}

In this work, we will always use the subscript to indicate elements
of a sequence and a superscript for the components of vectors
(which may belong to spaces other than Euclidean ones).
Denoting the volume element of $g$ by $\de v_g$,
we can define the (vectorial) volume functional
$V_g\colon L^1(M,\mathbb{R}^m) \to \mathbb{R}^m$ as follows:
\[
	V_g(u) \coloneqq
	\left(
		\int_{M} u^1 \de v_g,\,
		\dots,\,
		\int_{M} u^m \de v_g
	\right).
\]
Let us consider the Sobolev space
${W}^{1,2}(M,\mathbb{R}^m)\coloneqq\{u=(u^1,\dots,u^m) : {u}^i\in W^{1,2}(M,\mathbb{R}),
\; \forall i=1,\cdots,m\}$
equipped with the norm
\begin{equation*}
	\norm{u}_{W^{1,2}(M,\mathbb R^m)}:=\left(\sum_{i=1}^m\norm{u^i}^2_{{W}^{1,2}(M)}\right)^{1/2}.
\end{equation*}

The main goal of this paper is to study the existence and multiplicity
of solutions $u\in C^2(M,\mathbb{R}^m)$ to the following Allen-Cahn (AC) system
\begin{equation}
	\label{eq:ACH-PDE}\tag{$AC_{\varepsilon,\mathrm{v},m}$}
	\begin{dcases}
		-\varepsilon \Delta_g u
		+ \frac{1}{\varepsilon}\nabla W(u)
		=  \lambda
		& \text{ on } M,\\
		V_g(u) = \mathrm{v},
	\end{dcases}
\end{equation}
where $0<\varepsilon \ll1$ is the (small) temperature parameter,
$\lambda \in \mathbb{R}^m$ is a Lagrange multiplier associated
with the volume constraint
$\mathrm{v} \in \mathbb{R}^m_{> 0}\coloneqq \left\{ \mathrm{v} \in \mathbb{R}^m: \mathrm{v}^i >  0, \forall i = 1,\dots,m\right\}$,
and $W\in C^2(\mathbb{R}^m,\mathbb{R}_+)$ is a multi-well potential
(see Definition~\ref{def:multi-well-potential}), where
$\mathbb{R}_+ \coloneqq \big\{x \in \mathbb{R}: x \ge 0\big\}$.

System \eqref{eq:ACH-PDE} admits a variational formulation, which we describe as follows.
Let us consider the $C^2$--Hilbert manifold
\[
	W^{1,2}_{\mathrm{v}}(M,\mathbb{R}^m) \coloneqq
	\big\{
		u \in W^{1,2}(M,\mathbb{R}^m):
		V_g(u) = \mathrm{v}
	\big\},
\]
and define the functional
$\mathcal{E}_{\varepsilon,\mathrm{v}}\colon W^{1,2}_{\mathrm{v}}(M,\mathbb{R}^m) \to \mathbb{R}$ by
\begin{equation}
	\label{eq:def-calE-epsilon}
	\mathcal{E}_{\varepsilon,\mathrm{v}}(u)
	\coloneqq
	\int_{M} \Big(
		\frac{\varepsilon}{2}\norm{\nabla_g  u}^2
		+ \frac{1}{\varepsilon}W(u)
	\Big)\de v_g,
\end{equation}
where 
\[
	\norm{\nabla_g  u }^2
	\coloneqq
	\sum_{i = 1}^m
	g\big(\nabla_g u^i,\nabla_g u^i\big).
\]
A straightforward computation shows that~\eqref{eq:ACH-PDE} is the Euler--Lagrange equation
associated with \eqref{eq:def-calE-epsilon};
in other words, if a function $u \in W^{1,2}_{\mathrm{v}}(M,\mathbb{R}^m)$
is a critical point of $\mathcal{E}_{\varepsilon,\mathrm{v}}$,
then it is a solution to~\eqref{eq:ACH-PDE}. 
For the sake of clarity,
let us explicitly write the Fr\'echet derivative of $\mathcal{E}_{\varepsilon,\mathrm{v}}$. 
Since $W^{1,2}_{\mathrm{v}}(M,\mathbb{R}^m)$ is defined by a volume constraint,
the set of admissible variations is independent of $\mathrm{v} \in \mathbb{R}^{m}_{>0}$
and consists of the subspace of $W^{1,2}(M,\mathbb{R}^{m})$
with zero mean value on each component, namely
\[
	TW^{1,2}_{\mathrm{v}}(M,\mathbb{R}^m) \coloneqq
	\left\{\xi\in W^{1,2}(M,\mathbb{R}^m): \int_M \xi\,\de v_g = 0 \in \mathbb{R}^m
	\right\}.
\]
Thus, the first variation $\de\mathcal{E}_{\varepsilon,\mathrm{v}}( u)\colon TW^{1,2}_{\mathrm{v}}(M,\mathbb{R}^m)\to \mathbb{R}$ is given by
\[
	\de\mathcal{E}_{\varepsilon,\mathrm{v}}(u)[\xi]
	=
	\varepsilon\int_M \sum_{i = 1}^m
	g\big(\nabla_g u^i,\nabla_g \xi^i\big) \de v_g
	+ \dfrac1{\varepsilon}\int_M \sprod{\nabla W( u),\xi}
	\de v_g, 
\]
and if a function is a critical point of $\mathcal E_{\varepsilon,\mathrm{v}}$, the Lagrange multiplier theorem implies that it is a weak solution of \eqref{eq:ACH-PDE}.
\begin{remark}\label{remark11}
	If $u\in  W^{1,2}_{\mathrm{v}}(M,\mathbb{R}^m)$ is a weak solution
	of~\eqref{eq:ACH-PDE} and the gradient of the multi-well potential
	satisfies a subcritical growth condition
	(see~\eqref{eq:W1} in Definition~\ref{def:multi-well-potential}),
	then standard elliptic regularity theory implies that
	$u\in \mathcal C^{2,\gamma}(M,\mathbb{R}^m)$,
	for all $\gamma\in(0,1)$,
	making $u$ a classical solution of \eqref{eq:ACH-PDE}.
	This regularity result is obtained by applying a bootstrap argument,
	which employs $L^p$ and $\mathcal C^{0,\gamma}$ regularity
	(see \cite[Theorem 2.5]{MR888880}),
	together with Sobolev and Morrey embeddings theorems.
	Futhermore, if $W$ belongs to $C^\infty(\mathbb R^m,\mathbb R_+)$,
	we can infer that $u$ is smooth as well.
\end{remark}

In this manuscript, we establish a lower bound for the number of solutions
to~\eqref{eq:ACH-PDE}.
Our approach relies on the  
relation between the topology of a functional's sublevel sets
and its critical points,
a relation that is well-known in the mathematical community
and can be established using either
Morse theory or Lusternik--Schnirelmann theory.
We are able to estimate some topological 
invariants of a suitable sublevel set of our functional,
which is a subset of the infinite-dimensional space
$W^{1,2}_{\mathrm{v}}(M,\mathbb{R}^m)$,
based on the underlying manifold $M$ itself. 
As a consequence, we build a connection
between the topology of $M$ and the number of critical points 
of the functional $\mathcal{E}_{\varepsilon,\mathrm{v}}$ and
we establish a lower bound for the number of solutions to~\eqref{eq:ACH-PDE}
in terms of the two following 
topological invariants of the underlying manifold:
\begin{enumerate}
	\item[(i)] The Lusternik--Schnirelmann category of $M$, denoted by $\cat(M)$,
		which, as a reminder, is the smallest number of contractible
		(to a point) closed subsets required to cover $M$
		(see~\cite[Chapter 6]{Berger} for more details);
	\item[(ii)] The sum over all Betti numbers of $M$, which corresponds to the evaluation of the Poincar\'e polynomial of $M$ at $t=1$, denoted by $\mathcal{P}_1(M)=\sum_{k\in\mathbb N}\beta_k(M)$.
\end{enumerate}

Before presenting the statement 
of our main theorem, we need to introduce some additional terminology.
\begin{definition}
	Let $(M,g)$ be a closed Riemannian
	$N$-dimensional manifold with $N\geq2$.
	We say that $(M,g)$ is \textit{parallelizable}
	if there exist smooth vector fields
	$\{X_1, \dots, X_N\}\subset \Gamma(TM)$ such that at every point $p$ of $M$,
	the tangent vectors $\{X_1(p), \dots, X_N(p)\}\subset T_pM$ is a basis.
	Equivalently, the tangent bundle is trivial,
	so the associated principal bundle of linear frames has a global section on $M$.
	We say that a particular choice of such a basis of vector fields on $M$
	is called a \emph{parallelization} of $M$.
\end{definition}

\begin{remark}
	For $N=3$, spheres are parallelizable if and only if they are orientable,
	according to Stiefel’s theorem \cite{MR0440554}.
	Spheres $\mathbb{S}^N$ are parallelizable for $N=1,3,7$. 
	For $m=2$, this condition can be relaxed to require a smooth global section
	of the unit tangent bundle $UTM$,
	based on the weighted double bubble conjecture in $\mathbb{R}^N$
	and the pseudo-double-bubbles theory \cite{MR3145921}.
\end{remark}

\begin{definition}
	\label{def:multi-well-potential}
	A nonnegative function $W\in C^2(\mathbb{R}^m,\mathbb{R}_+)$ is called a \emph{multi-well potential} if it satisfies the following conditions:
	\begin{enumerate}
		\item[\namedlabel{item:W-global-minima}{($W_0$)}] $($\textit{immiscibility}$)$
			it has exactly $m + 1$ vanishing and nondegenerate global minima in $\mathbb{R}^m$, denoted by $\mathcal{Z}\coloneqq\{\mathbf{z}_{0},\dots,\mathbf{z}_{m}\}\subset\mathbb{R}^m$;
			this means that, for each $i=0,\ldots,m$, we have $W(\mathbf{z}_i)=0$, $\nabla W(\mathbf{z}_i)=0$, and $\nabla^2W(\mathbf{z}_i)>0$.
			We assume that $\mathbf{z}_0=0$,
			and $\{\mathbf{z}_{1},\dots,\mathbf{z}_{m}\}$
			is linearly independent.
			Moreover, the following \emph{strict} inequality holds:
			\begin{equation*}
				\tag{${{\rm W}_0}$}\label{eq:W0}
				\omega_{ij} <
				\omega_{i\ell} + \omega_{\ell j},
				\quad \forall i,j,\ell\in\{0,\dots, m\},\; \ell\notin\{i,j\},
			\end{equation*}
			where the $\omega_{ij}$-s are defined as follows:
			\begin{equation}
				\label{eq:def-omega-ij}
				\omega_{ij} 
				\coloneqq
				\inf_{\gamma\in \Gamma_{ij}}
				\int_{0}^{1}W^{1/2}(\gamma(s))\norm{\gamma^{\prime}(s)}\de s,
			\end{equation}
			and
			\begin{equation*}
				\Gamma_{ij} = \left\{
					\gamma \in C^1([0,1],\mathbb{R}^m):
										\gamma(0)=\mathbf{z}_i, \; \gamma(1)=\mathbf{z}_j
				\right\};
			\end{equation*} 
		\item[\namedlabel{item:W-def-k1}{{($W_1$)}}]
			\textit{$($gradient growth condition$)$} there exists a constant $k_1>0$ such that 
			\begin{equation*}
				\tag{${{\rm W}_1}$}\label{eq:W1}
				\norm{\nabla W(z)} \le k_1(1+\norm{z}^{p-1}),
				\quad \forall z\in\mathbb{R}^m,
			\end{equation*}
			for some $1<p<2^{*}$ if $N\ge 3$ (or $p<\infty$ if $N=2$), where $2^{*}\coloneqq\frac{2N}{N-2}$ is the critical Sobolev exponent of the embedding $W^{1,2}(\mathbb R^N)\hookrightarrow L^q(\mathbb R^N)$ for $q>1$;
		\item[\namedlabel{item:W-def-k2}{{($W_2$)}}]
			$($\textit{Hessian growth condition}$)$ there exists $k_2>0$ such that 
			\begin{equation}
				\tag{${{\rm W}_2}$}\label{eq:W2}
				\norm{\nabla^2 W(z)}\le k_2(1+\norm{z}^{p-2}),
				\quad \forall z\in\mathbb{R}^m
			\end{equation}
			for some $1<p<2^{*}$ if $N\ge 3$ (or $1<p<\infty$ if $N=2$);
		\item[\namedlabel{item:W-def-p12-k34-R}{($W_3$)}]
			$($\textit{growth at infinity}$)$ there exist $p_1, p_2, k_3, k_4,R>0$ such that
			\begin{equation}\tag{${{\rm W}_3}$}
				\label{eq:W3}
				k_3\norm{z}^{p_1}<W(z)<k_4\norm{z}^{p_2},
				\quad \forall z\in\mathbb{R}^m\setminus B(0,R),
			\end{equation}  
			where $2<p_1<2^{\#}$ with $p_1\le p_2\le 2(p_1-1)$ and $2^{\#}\coloneqq\frac{2N-1}{N-1}$.
	\end{enumerate}
\end{definition}

\begin{figure}
    \centering
    \begin{tikzpicture}
        \node[anchor=south west,inner sep=0] (image) at (0,0) {\includegraphics[width=0.4\textwidth]{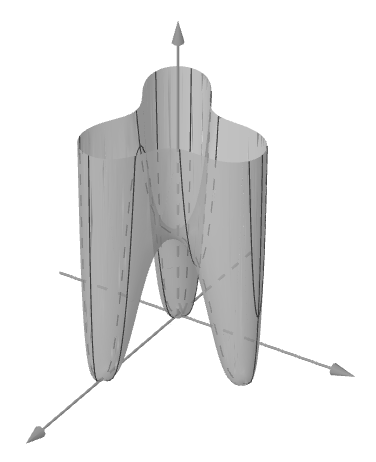}};
        \begin{scope}[x={(image.south east)},y={(image.north west)}]
            \node at (0.48,0.27) {$\mathbf{z}_0$};
            \node at (0.29,0.13) {$\mathbf{z}_1$};
            \node at (0.65,0.13) {$\mathbf{z}_2$};
            \node at (0.75,0.75) {$W$};
        \end{scope}
    \end{tikzpicture}
    \caption{A graphic representation of a multi-well potential (see Definition~\ref{def:multi-well-potential}), where $m = 2$ and the points $\{\mathbf z_0,\mathbf z_1,\mathbf z_2\}$ form the vertices of an equilateral triangle.}
    \label{fig:potential-well}
\end{figure}
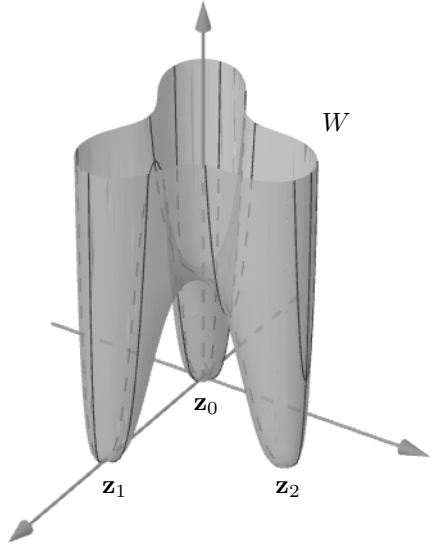

\begin{example}
	Assume that $\mathbf{z}_0 = 0$ and that
	$\{\mathbf{z}_{1},\dots,\mathbf{z}_{m}\}$
	are positioned such that $\norm{\mathbf z_i-\mathbf z_j}= 1$ if $i\ne j$.
	In other words, $\mathbf z_0,\ldots\mathbf z_m$
	are the vertices of a regular $m$-simplex.
	If $p\geq2m+2$  and
	$p\in (p_1,p_2)$,
	where $p_1,p_2\in(2,\infty)$
	are defined in~\ref{item:W-def-p12-k34-R},
	it can be verified that the function
	$W(z)=(\prod_{i=0}^m\norm{z-\mathbf{z}_i})^{p/(m+1)}$,
	represents a $C^2$ multi-well potential satisfying our required conditions.
	Figure~\ref{fig:potential-well} shows a graph of this function for $m=2$. 
	Furthermore, due to the symmetry of the set
	$\{\mathbf z_0,\ldots,\mathbf z_m\}$ and the potential $W$,
	we have that $\omega_{ij}=\ell$ if $i \ne j$,
	where $\ell > 0$ is given by~\eqref{eq:def-omega-ij},
	hence we have
	\[
		\omega = 
		\ell
		\begin{bmatrix}
			0 & 1 & 1 & \dots & 1 \\
			1 & 0 & 1& \dots & 1 \\
			1 & 1 & 0 &\dots & 1 \\
			\vdots & \vdots & \vdots & \ddots & \vdots \\
			1 & 1 & 1 &\dots & 0
		\end{bmatrix}
		\in \mathbb{R}^{(m+1) \times (m+1)}.
	\]
\end{example}

\begin{remark}\label{remark:w0}
	By the definition of $\omega_{ij}$ in~\eqref{eq:def-omega-ij},
	one can notice that
	the inequality~\eqref{eq:W0} is always valid if one considers
	the non-strict version of the inequality,
	replacing ``less than" with ``less than or equal to".
	As observed in \cite{MR1402391},
	this can be interpreted as a physical hypothesis
	concerning the fluids' immiscibility.
	In this context, immiscibility refers to the inability
	of two or more substances to mix into a homogeneous mixture.
	When such substances are combined,
	they separate into distinct layers or phases rather
	than creating a uniform mixture.
	This typically occurs due to differences in molecular structure,
	polarity, or intermolecular forces between the substances.
	In addition, the condition that the global minima
	of the potential well are nondegenerate 
	is necessary to ensure that exactly $m$-phases will be formed.     Throughout the paper, we assume the strict inequality in \eqref{eq:W0} to ensure that the interfaces are smooth hypersurfaces. For further details on this regularity result, we refer to \cite[Theorem 2]{MR1402391}.
\end{remark}

To state the main theorem,
we formalize the notion of
nondegenerate solution to~\eqref{eq:ACH-PDE}.
\begin{definition}
	We say that a solution
	$(u,  \lambda) \in W_\mathrm{v}^{1,2}(M,\mathbb{R}^m) \times \mathbb{R}^m$
	to \eqref{eq:ACH-PDE} is \emph{nondegenerate}
	if the only pair
	$(\xi,\lambda) \in W^{1,2}(M,\mathbb{R}^m)\times \mathbb{R}^m$
	that solves the linearized problem below
	\begin{equation}
		\label{eq:linearizedsystem}
\de^2\mathcal{E}_{\varepsilon,\mathrm{v}}(u)[\xi] = \lambda
				\quad \text{on } M,
	\end{equation}
	is the trivial one $(\xi,  \lambda)=(0,0)$,
	where $\de^2\mathcal{E}_{\varepsilon,\mathrm{v}}(u)[\xi]
	=-\varepsilon \Delta_g \xi+ \varepsilon^{-1}\nabla^2 W(u)\xi$.
\end{definition}

Now, we are ready to state our main result.
\begin{theoremletter}
	\label{theorem:main}
	Let $N$ and $m$ be two integers such that 
	$N \ge 2$ and $m \ge 1$.
	Let $(M,g)$ be an $N$-dimensional 
	closed and parallelizable Riemannian manifold
	let $W\in C^2(\mathbb{R}^m,\mathbb{R}_+)$
	a multi-well potential.
	For every $\mathrm{v} \in \mathbb{R}^m_{>0}$
	there exists $\alpha^* = \alpha^*(M,g,\mathrm{v}) > 0$
	such that for every $\alpha \in (0,\alpha^*)$
	there exists $\varepsilon^* = \varepsilon^*(\alpha,\mathrm{v}) > 0$
	such that for every $\varepsilon \in (0,\varepsilon^*)$,
	the functional $\mathcal E_{\varepsilon,\alpha\mathrm v}\colon W^{1,2}_{\alpha\mathrm{v}}(M,\mathbb{R}^m)\to\mathbb R$ admits:
	\begin{itemize}
		\item[{\rm (i)}] at least $\cat(M)+ 1$
			distinct critical points;
		\item[{\rm (ii)}] at least $2\mathcal{P}_1(M) - 1$
			distinct critical points if they are all nondegenerate.
	\end{itemize}
\end{theoremletter}

\begin{remark}
	The assumption \ref{item:W-def-k2} implies that the property that
	\quotes{all critical points of $\mathcal E_{\varepsilon,\alpha\mathrm v}\colon W^{1,2}_{\alpha\mathrm{v}}(M,\mathbb{R}^m)\to\mathbb R$ are nondegenerate}
	is generic (open and dense) on the space of
	(smooth) $C^\ell$-metrics on $M$,
	with $\ell\in\mathbb{N}$ sufficiently large.
	Thus, the lower bound $2\mathcal{P}_1(M) - 1$ holds generically.
	The proof of this fact is addressed in \cite[Section~6]{MR4701348} when $m\geq 2$ and in \cite{MR4314216}.
	Both proofs are based on a generalization of Smale's infinite-dimension version of Sard's theorem (see \cite[Theorem~5.6]{MR2160744}).
	For more references on genericity results in geometric-variational problems, we refer to~\cite{MR2560131,MR1402391,MR4427104,arXiv:2407.06934}.
\end{remark}

Let us compare our main result with those of some related works.
In~\cite{MR4644903},
V. Benci and some of the authors
obtained multiplicity results for relaxed isoperimetric sets
enclosing a small volume,
thus providing an analogous result for the case $m = 1$.
For the sake of the reader's convenience,
we report here their main result.
\begin{statement}{Theorem (cf.~\cite[Theorem 2.1]{MR4644903})}
	Let $(M,g)$ be an $N$-dimensional closed Riemannian manifold
	with $N\ge 2$ and let $W\in C^2(\mathbb{R},\mathbb{R}_+)$
	be a double-well potential.
	There exists $\mathrm{v}^*=\mathrm{v}^*(M,g)>0$ such that for every $\mathrm{v}\in(0,\mathrm{v}^*)$, there exists $\varepsilon^*(\mathrm{v}^*)>0$ such that for every $\varepsilon\in(0,\varepsilon^*)$	the functional $\mathcal E_{\varepsilon,\mathrm v}\colon W^{1,2}_{\mathrm{v}}(M,\mathbb{R})\to\mathbb R$
	admits:
	\begin{itemize}
		\item[{\rm (i)}] at least $\cat(M)+ 1$ distinct critical points;
		\item[{\rm (ii)}] at least $2\mathcal{P}_1(M) - 1$  distinct critical points if they are all nondegenerate.
	\end{itemize}
\end{statement}
Again, for the case of just one small volume, {\it i.e.}, $m=1$,
in~\cite{MR4073210}, a similar multiplicity result was obtained
when $M$ is the closure of an open subset of the Euclidean space
$\mathbb{R}^{N}$, with the difference that $W$ was a 
non--symmetric double well potential.

The first work to address the vectorial case $m=2$ is~\cite{MR4701348},
where a similar multiplicity result is established
for relaxed isoperimetric two-clusters enclosing a small volume,
leading to the following theorem.
\begin{statement}{Theorem (cf.~\cite[Theorem A]{MR4701348})}
	Let $(M,g)$ be an $N$-dimensional 
	closed and parallelizable Riemannian manifold
	with $N\ge 2$ and let $W\in C^2(\mathbb{R}^2,\mathbb{R}_+)$
	a triple-well potential.
	There exists $\mathrm{v}^*=\mathrm{v}^*(M,g)>0$ such that for every $\mathrm{v}\in(0,\mathrm{v}^*)\times (0,\mathrm{v}^*)$, there exists $\varepsilon^*(\mathrm{v}^*)>0$ such that for every $\varepsilon\in(0,\varepsilon^*)$	the functional $\mathcal E_{\varepsilon,\mathrm v}\colon W^{1,2}_{\mathrm{v}}(M,\mathbb{R}^2)\to\mathbb R$ admits:
	\begin{itemize}
		\item[{\rm (i)}] at least $\cat(M)+ 1$ distinct critical points;
		\item[{\rm (ii)}] at least $2\mathcal{P}_1(M) - 1$  distinct critical points if they are all nondegenerate.
	\end{itemize}
\end{statement}

Their proof is a combination of $\Gamma$--convergence results of the Allen--Cahn energy to the perimeter functional \cite{MR866718,MR930124}, and diameter estimates for \quotes{almost} isoperimetric regions of small volume
(see, e.g.,~\cite{MR4467099,MR4745750,MR2529468}).
Moreover, for the case of $m = 1$ 
and $m = 2$,
the proofs also depend on the explicit classification of minimizers
of the isoperimetric profile,
which are the bubbles if $m = 1$ and,
in the case of weighted two-clusters,
the so-called weighted double-bubbles,
as proved in \cite{MR3145921}.

Compared with~\cite[Theorem 2.1]{MR4644903} and~\cite[Theorem A]{MR4701348},
Theorem~\ref{theorem:main} needs to introduce a new parameter $\alpha^*>0$,
the so-called {\it order rescaling parameter},
which ensures that the chamber volumes are small
\emph{in the same order}.
This is due to the lack of a full classification
of minimizers of the weighted isoperimetric profile for clusters
with $m \ge 3$, hence the same proof that lead to the multiplicity 
result for $m = 1,2$ doesn't work.
To address this challenge,
we rely on a weighted version of a fixing volume
variation technique (Lemma~\ref{lem:fixing-volumes}).
This procedure seeks a way to modify (through local variations) a cluster
in order to achieve an arbitrary (but suitably small) change
in the volumes of its chambers, with a controlled perimeter increase.
However, once a cluster has been fixed, we are able to 
ensure that the small variations are sufficient for our needs only 
when we ``rescale'' the problem through a scalar parameter,
and this leads to the presence of $\alpha^* > 0$
(see Section~\ref{sec:setting} for more details).
Geometrically, the price to pay for increasing the number of equations in
system~\eqref{eq:ACH-PDE} is that instead of imposing
$\mathrm{v}\in (0,\mathrm{v}^*)^m$ for some $\mathrm{v}^*>0$,
we require that, for any fixed proportions of the chambers,
namely for any fixed $\mathrm{v} \in \mathbb{R}_{>0}^m$,
we can impose the total volume sufficiently small
to obtain the final desired multiplicity result.

\smallskip
The outline for the rest of the paper is as follows.
In \S~\ref{sec:tools}, we present some tools and preliminaries;
namely, we describe the abstract photography method,
the multi-isoperimetric problem,
and $\Gamma$--convergence.
The photography method, which is the basis for the 
proof of Theorem~\ref{theorem:main},
requires the construction of a \emph{photography map},
defined on the manifold and with values in the Sobolev space
of volume-constrained functions,
and of a \emph{barycenter map},
defined on the smallest sublevel that contains the image
of the photography map and with values in the manifold.
In \S~\ref{sec:setting},
we introduce some useful notation and
sketch the main ideas to construct
both the photography map 
and the barycenter map.
In \S~\ref{sec:photographymap},
we provide all the details for the construction of the photography map;
proving also prove its continuity and 
estimating of the smallest sublevel that contains its image.
In \S~\ref{sec:barycenter} we prove that a 
continuous intrinsic barycenter map can be defined 
on that sublevel and, finally,
the proof of Theorem~\ref{theorem:main} is given in
\S~\ref{sec:finalproof}.

For the reader's convenience,
a list of notation is provided at the end of the paper.

\section{Main Tools and Preliminaries}
\label{sec:tools}

\subsection*{Abstract photography method}
\label{subsec:abstarctphotography}
Our multiplicity result relies on the so-called \emph{photography method},
which, loosely speaking, is a general framework that allows to estimate
the topology of a sublevel set of a functional based on an underlying manifold.
This method was first introduced by V. Benci and G. Cerami
in several papers in the early 1990s
(see the classical ones \cite{MR1088278,MR1384393}),
and has been employed to prove the existence of multiple solutions
for various elliptic problems.
Notable applications include the work by
S. Cingolani and M. Lazzo~\cite{MR1646619,MR1734531} on standing waves
of a nonlinear Schr\"odinger equation,
as well as the recent studies
of J. Petean~\cite{MR3912791},
where a multiplicity result for the Yamabe equation is obtained,
and of S. Alarcón, J. Petean, and C. Rey~\cite{MR4761862},
which explores the multiplicity of conformal metrics
with constant $Q$-curvature.
In recent years, this method has also been successfully applied
to find multiple solutions to the
AC system~\cite{MR4701348,MR4396580,MR4073210,arXiv:2401.17847}.
Roughly speaking, one can hope to apply such a method when
the functions that lie in a sublevel set of the functional
are expected to be concentrated in a small portion of the manifold.

For any $c \in \mathbb{R}$,
let us denote by $\mathcal{E}_{\varepsilon,\mathrm{v}}^c$
the $c$--sublevel set of our functional, namely
\[
	\mathcal{E}_{\varepsilon,\mathrm{v}}^c
	\coloneqq
	\left\{ u \in W^{1,2}_{\mathrm{v}}(M,\mathbb{R}^m) :
	\mathcal{E}_{\varepsilon,\mathrm{v}}(u) \leq c \right\}.
\]
To effectively apply the photography method,
it is necessary to have 
an energy level $c \in \mathbb{R}$
and two continuous functions,
\[
	\varphi\colon M \to \mathcal{E}_{\varepsilon,\mathrm{v}}^c
	\quad \text{and} \quad
	\beta\colon \mathcal{E}_{\varepsilon,\mathrm{v}}^c \to M,
\]
such that their composition $\beta \circ \varphi \colon M \to M$ 
is homotopic to the identity map.
For instance, $\varphi$ cannot be a constant map; specifically, $\varphi(M) \subset \mathcal{E}_{\varepsilon,\mathrm{v}}^c$ is expected to faithfully preserve all topological properties of $M$, thereby justifying the term \emph{photography} for both the method and the map $\varphi$.

From an abstract point of view,
the general statement of the photography method
is the following,
and its proof can be found in~\cite[Section 2]{MR4701348}
\begin{theoremletter}[\cite{MR1322324,MR1384393}]
	\label{theorem:abstract-photography}
	Let $X$ be a topological space,
	$\mathfrak{M}$ be a $C^2$-Hilbert manifold,
	$\mathcal{E}\colon \mathfrak{M}\to\mathbb{R}$ be a $C^1$-functional,
	and $\mathcal{E}^c\coloneqq\{u\in\mathfrak{M} : \mathcal{E}(u)\le c\}$
	be a sublevel set for some $c\in\mathbb R$.
	Assume that the following properties hold:
	\begin{enumerate}
		\item[\namedlabel{itm:E1}{($E_1$)}]
			$\inf_{u\in\mathfrak{M}}\mathcal{E}(u)>-\infty$;
		\item[\namedlabel{itm:E2}{($E_2$)}]
			$\mathcal{E}$ satisfies the Palais--Smale (PS) condition;
		\item[\namedlabel{itm:E3}{($E_3$)}]
			There exist $c\in\mathbb R$ and two continuous maps
			$\Psi_{R}\colon X\rightarrow \mathcal{E}^c$ and
			$\Psi_{L}\colon\mathcal{E}^c\rightarrow X$
			such that $\Psi_{L}\circ\Psi_{R}$ is homotopic
			to the identity map of $X$.
	\end{enumerate} 
	Then, the number of critical points in $\mathcal{E}^c$ 
	is greater than $\cat(X)$,
	and if $\mathfrak{M}$ is contractible and $\cat(X)>1$,
	there exists at least another critical point of $\mathcal{E}$ outside
	$\mathcal{E}^c$.
	Moreover,
	if all the critical points are nondegenerate,
	there exists $c_0\in(c,\infty)$ such that 
	$\mathcal{E}^{c}$ contains $\mathcal{P}_1(X)$ critical points
	and $\mathcal{E}^{c_0}\setminus \mathcal{E}^{c}$
	contains $\mathcal{P}_1(X)-1$ critical points
	if counted with their multiplicity.
	More precisely, the following relation holds:
	\begin{equation}
		\label{eq:morserelation}
		\sum_{u\in {\rm Crit(\mathcal E)}} t^{\mu(u)}
		=\mathcal{P}_t(X)+t[\mathcal{P}_t(X)-1]+(1+t)\mathcal{Q}(t),
	\end{equation}
	where $\mathcal{Q}(t)$ is a polynomial
	with nonnegative integer coefficients,
	${\rm Crit}(\mathcal{E})$
	denotes the set of critical points of $\mathcal{E}$
	and $\mu(u)$ denotes the (numerical) Morse index
	of $u$, {\it i.e.}, the dimension of the maximal subspace
	on which the bilinear form
	$\mathrm{d}^2\mathcal{E}_{\varepsilon}(u)[\cdot,\cdot]$
	is negative-definite.
\end{theoremletter}
\begin{remark}
	It is easy to see that $W^{1,2}_\mathrm{v}(M)$
	is a contractible topological space.
	Indeed, it can be continuously sent to the constant vectorial function
	$\bar{u}\equiv\frac{\mathrm{v}}{\vol_g(M)}$.
	Moreover, any closed Riemannian manifold $M$
	is not contractible in itself,
	hence $\cat(M)>1$ 
	and Theorem~\ref{theorem:abstract-photography}
	provides a critical point with 
	``high'' energy.
	Actually, one solution of high energy is
	the constant function $\bar{u}$,
	since it trivially satisfies~\eqref{eq:ACH-PDE}.
\end{remark}

Since $\mathcal{E}_{\varepsilon,\mathrm{v}}$ is trivially
lower-bounded and satisfies the Palais-Smale condition,
as demonstrated in \cite[Lemma 5.5]{MR4701348},
we can directly infer Theorem~\ref{theorem:main}
from Theorem~\ref{theorem:abstract-photography} if we are able to:
\begin{enumerate}
	\item[\namedlabel{item:s1}{({\it Step 1})}]
		construct a continuous photography map
		$\varphi_{\varepsilon,\mathrm{v}}\colon
		M \to W^{1,2}_{\mathrm{v}}(M,\mathbb{R}^m)$ and
		give and estimate  of a sublevel that contains its image,
		namely provide $c\in \mathbb{R}$ such that
		$\varphi_{\varepsilon,\mathrm{v}}(M)\subset \mathcal{E}^c_{\varepsilon,\mathrm{v}}$;
	\item[\namedlabel{item:s2}{({\it Step 2})}]
		construct a continuous ``barycenter'' map
		$\beta: \mathcal{E}^c_{\varepsilon,\mathrm{v}}\rightarrow M$; 
	\item[\namedlabel{item:s3}{({\it Step 3})}]
		prove that the composition
		$\beta\circ\varphi_{\varepsilon,\mathrm{v}} \in C^0(M, M)$
		is homotopic to identity.
\end{enumerate}

For all the above steps of the proof,
we rely on the $\Gamma$--convergence of our functional, 
as $\varepsilon\rightarrow 0$, to the multi-isoperimetric problem, 
and to the fact that,
if $\mathrm{v}\in \mathbb{R}^m_{>0}$ is sufficiently small, 
then the ``almost minimizers'' functions of
$\mathcal{E}_{\varepsilon,\mathrm{v}}$ 
must have \emph{almost} all their mass in a small ball, 
where even this last property will be proved 
by using an analogous concentration property 
of the ``almost minimizers'' of the multi-isoperimetric 
problem and the $\Gamma$--convergence.
Hence, let us introduce the 
multi-isoperimetric problem and the $\Gamma$--convergence
results.

\subsection*{Multi-isoperimetric problem}
\label{subsec:multi-isoperimetric}
System~\eqref{eq:ACH-PDE} is strictly linked
with the multi-isoperimetric problem,
that consists in minimizing the multi-perimeter functional 
among all the $m$-clusters of $M$ that enclose 
a fixed vectorial volume.
This relation is given by the notion of $\Gamma$--convergence.

Let us introduce the following notation to state formally 
the multi-isoperimetric problem.
For every $\Omega \subset M$, let $|\Omega|$ denote its Lebesgue measure,
$\chi_{\Omega}$ its characteristic function,
$\mathcal{H}_g^{N-1}(\Omega)$ its $(N-1)$--Hausdorff measure and
$\partial^*\Omega$ its reduced boundary.

Let $\mathcal{C}_g^m(M)$ be the class of all collections of 
$m$-finite perimeter subsets of $M$ without intersections, namely
\begin{equation}
	\label{eq:def-CgmM}
	\mathcal{C}_g^m(M)
	\coloneqq \left\{
		\Omega = (\Omega^1,\dots,\Omega^{m}):
		\begin{aligned}
			\Omega^i\subset M  \text{ is open },
			\, 
			& \Omega^i \cap \Omega^j = \emptyset, \text{ and } \mathcal{H}^{N-1}_g(\partial^* \Omega^i)< \infty
			\\
			& 
			\forall i,j=1,\dots,m \text{ with } i \ne j.
		\end{aligned}
	\right\}.
\end{equation}
Each element of $\mathcal{C}_g^m(M)$ will be called 
\emph{$m$-cluster}
(or, more simply, \emph{cluster}) and 
each component of a cluster will be called \emph{a chamber}.
For every $\Omega\in \mathcal{C}_g^m(M)$,
we denote by $\Omega^0={\rm int}(M \setminus \bigcup_{i = 1}^m \Omega^i)$ its \emph{exterior} chamber.
We say that a point $p\in M$
belongs to a cluster $\Omega \in \mathcal{C}^m_g(M)$,
and we write $p \in \Omega$,
if it is in one of its interior chambers,
hence if $p \in \Omega^i$ for some $i = 1,\dots,m$.
With this notation, we can define the
\emph{diameter} of a cluster as the 
supremum of the distances between two of its points, namely
\begin{equation}
	\label{eq:def-diam-cluster}
	\diam(\Omega)\coloneqq
	\sup\Big\{
		\dist_g(p,q):
		p,q \in \Omega
	\Big\}.
\end{equation}
We define the symmetric difference between two 
clusters as the cluster obtained by the symmetric 
differences of their chambers:
\[
	\Omega_1\triangle\Omega_2
	\coloneqq
	\big(\Omega_1^1\,\triangle\,\Omega_2^1,\,
		\dots,
	\Omega_1^m\,\triangle\,\Omega_2^m\big)\,
	\in \mathcal{C}_g^m(M),
	\qquad \forall \Omega_1,\Omega_2 \in \mathcal{C}^m_g(M).
\]
Consequently,
we define the distance function
$\dist\colon \mathcal{C}^m_g(M)\times \mathcal{C}^m_g(M) \to [0,\infty)$
as follows:
\begin{equation}
	\label{eq:dist-on-clusters}
	\dist(\Omega_1,\Omega_2)
	= \sum_{i = 1}^m
	|\Omega_1^i \triangle \Omega_2^i|.
\end{equation}
In other words, it is the sum of the Lebesgue measures
of the symmetric differences of the chambers.

For our purposes, we need to associate 
to every $m$-cluster
the $\mathbb{R}^m$ valued function that is 
equals to $\textbf{z}_i$ on the $i$--th chamber,
and $0$ otherwise, where we recall 
that $\{\mathbf{z}_1,\dots,\mathbf{z}_m\}$
are the non-trivial zeros of the potential function $W$
(see Definition~\ref{def:multi-well-potential}).
Hence, for every $\Omega\in \mathcal{C}^m_g(M)$
we define $Z_\Omega \in L^1(M,\mathbb{R}^m)$ as 
\begin{equation}
	\label{eq:def-X_Omega}
	Z_\Omega(p) \coloneqq \sum_{i= 1}^m \mathbf{z}_{i}\chi_{\Omega^i}(p).
\end{equation}
Let us notice that $W\circ Z_\Omega$ vanishes 
a.e. on the whole manifold,
but $Z_\Omega$ is not a Sobolev function,
due to the step discontinuities that occur 
in the interfaces between the chambers
(excluding, of course, the trivial cases, such as
$\Omega = \emptyset$).

We define the map $\mathcal{V}_g\colon \mathcal{C}_g^m(M) \to \mathbb{R}^m$,
which, loosely speaking, is the vectorial volume of clusters,
as follows:
\[
	\mathcal{V}_g(\Omega)\coloneqq
	\left(
		\int_{M}
		\chi_{\Omega^1}\de v_{g},
		\dots,
		\int_{M}
		\chi_{\Omega^m}\de v_{g}
	\right)
	\in \mathbb{R}^{m}_{+}.
\]
For every $\mathrm{v}\in \mathbb{R}^m_{>0}$,
we denote by $\mathcal{C}^m_{g,\mathrm{v}}(M)$
the subset of $\mathcal{C}^m_g(M)$
whose elements have vectorial volume $\mathrm{v}$,
namely
\begin{equation}
	\label{eq:def-CgmMv}
	\mathcal{C}^m_{g,\mathrm{v}}(M)
	\coloneqq
	\Big\{
		\Omega \in \mathcal{C}_g^m(M):
		\int_{M} \chi_{\Omega^i}\de v_{g} = \mathrm{v}^i, \;
		\forall i = 1,\dots, m
	\Big\}.
\end{equation}
Let $\omega$ denote the $(m+1)$--dimensional symmetric matrix 
whose coefficients $\omega_{ij}$ are given by~\eqref{eq:def-omega-ij}
and let us define 
the $\omega$-weighted multi-perimeter functional
$\mathcal{P}^{\omega}_g\colon \mathcal{C}_g^m(M) \to \mathbb{R}$
as follows:
\begin{equation}
	\label{eq:def-Pomega}
	\mathcal{P}^\omega_g(\Omega)
	\coloneqq
	\frac{1}{2}
	\sum_{i,j = 0}^{m}
	\omega_{ij}
	\mathcal{H}^{N-1}_g(\partial^*\Omega^{i}\cap \partial^*\Omega^j).
\end{equation}
The \emph{weighted multi-isoperimetric problem} asks to minimize
$\mathcal{P}^\omega_{g}$ among all the 
clusters with a specified vectorial volume.
The $\omega$-weighted multi-isoperimetric function
$\mathcal{I}^{\omega}_g\colon (0,{\rm vol}_g(M))^{m} \to \mathbb{R}$
associates at each vectorial volume this minimal value,
hence it is
\begin{equation}
	\label{eq:def-Iomega}
	\mathcal{I}^{\omega}_{g}(\mathrm{v})
	\coloneqq
	\inf\big\{
		\mathcal{P}^{\omega}_g(\Omega):
		\Omega \in \mathcal{C}^m_{g,\mathrm{v}}(M)
	\big\}.
\end{equation}

\begin{remark}\label{remarkhere}
	The infimum $\mathcal{I}^{\omega}_{g}(\mathrm{v})$ is always attained,
	and the $m$-clusters that achieve it are known as {\it isoperimetric $\omega$-weighted} $m$-clusters.
	F. Almgren gave this existence result for the non-weighted case in his classical work~\cite[Chapter VI]{MR420406}. His proof is based on a clever application of the direct method of calculus of variations, hence by considering a minimizing sequence for the multi-perimeter functional, proving that this sequence satisfies the compactness criterion for clusters
	(see~\cite[Proposition 29.5]{MR2976521} for a modern exposition of this compactness result)
	and then relying on the lower semicontinuity of the perimeter functional.
	When different weights $\omega$ are considered, provided they satisfy the immiscibility condition given by~\eqref{eq:W0}, the results in~\cite{MR1070482} ensure that $\mathcal{P}^\omega_g$ is still lower semicontinuous. This guarantees that $\mathcal{I}^{\omega}_g(\mathrm{v})$ is indeed attained. Moreover, by \cite[Theorem 2]{MR1402391} and \cite[Theorem 3.1]{leonardi2001}, the interfaces between the chambers of an isoperimetric $\omega$-weighted $m$-cluster $\Omega$, {\it i.e.} $\partial^* \Omega^i \cap \partial^* \Omega^j$, are smooth.
\end{remark}

All the above definitions, namely
$\mathcal{C}^m_g(M)$, $\mathcal{P}^{\omega}_g$ and $\mathcal{I}^{\omega}_g$,
can also be given in the $N$-dimensional
euclidean setting, 
replacing $M$ with $\mathbb{R}^N$.
The corresponding notions will be denoted 
replacing $\delta$ with $g$,
so they are 
$\mathcal{C}^m_{\delta}$, $\mathcal{P}^{\omega}_{\delta}$
and $\mathcal{I}^{\omega}_{\delta}$, respectively.

\begin{remark}
    \label{rem:bounded-isoperimetric-clusters}
    Following the proof of~\cite[Lemma 3.22]{MR4701348}
    (see also~\cite[Theorem 3]{MR4588150}
    and~\cite[Lemma 13.6]{MR3497381}),
    it can be shown that,
    under the immiscibility condition~\eqref{eq:W0},
    for every $\mathrm{v}\in\mathbb{R}^m_{>0}$
    there exists a \emph{bounded} isoperimetric
    $\omega$-weighted $m$-cluster
    in the euclidean setting.
    Namely, there exists
    $\Omega_{\mathrm{v}} \subset \mathcal{C}_\delta^m(\mathbb{R}^N)$
    and $R > 0$
    such that 
    \[
        \mathcal{P}_\delta^\omega(\Omega_{\mathrm{v}}) = 
        \mathcal{I}^{\omega}_{\delta}(\mathrm{v}),
        \quad\text{and}\quad
        \bigcup_{i = 1}^m \Omega_{\mathrm{v}}^i
        \subset B(0,R).
    \]
\end{remark}

\subsection*{\texorpdfstring{$\Gamma$}{}--convergence results}
\label{subsec:gamma-convergence}

We formalize the fact that,
when $\varepsilon$ is infinitesimal,
the system’s energy approximates the $\omega$--weighted
multi perimeter functional,
leading to solutions that concentrate in regions
resembling multi-isoperimetric clusters.

The main result of this part is stated below.
\begin{proposition}
	\label{prop:gamma-convergence}
	For any $\mathrm{v} \in \mathbb{R}^m_{> 0}$
        such that $\sum_{i = 1}^m \mathrm{v}^i < \vol(M)$,
	the sequence of functionals
	$(\mathcal{E}_{\varepsilon,\mathrm{v}})_{\varepsilon > 0}$ 
	$\Gamma$--converges to the weighted perimeter functional
	$\mathcal{P}^{\omega}_{g}\colon \mathcal{C}^m_{g,\mathrm{v}}(M) \to \mathbb{R}$,
	meaning that the following properties hold:
	\begin{itemize}
		\item[{\rm (i)}]
			$($\emph{liminf property}$)$ for any $\Omega \in \mathcal{C}^m_{g,\mathrm{v}}(M)$
			and family
			$(u_{\varepsilon})_{\varepsilon>0}\subset W^{1,2}_\mathrm{v}(M,\mathbb{R}^m)$
			that converges to $Z_\Omega \in L^1(M,\mathbb{R}^m)$
			$($where $Z_\Omega$ is given by~\eqref{eq:def-X_Omega}$)$,
			namely
			$\lim_{\varepsilon \to 0}
			\norm{u_\varepsilon - Z_\Omega}_{L^1(M,\mathbb{R}^m)}=0$,
			we have
			\begin{equation}
				\label{eq:liminf-prop}
				\liminf_{\varepsilon \to 0}
				\mathcal{E}_{\varepsilon,\mathrm{v}}(u_\varepsilon)
				\ge \mathcal{P}^\omega_{g}(\Omega);
			\end{equation}
		\item[{\rm (ii)}]
			$($\emph{limsup property}$)$ for any $\Omega\in \mathcal{C}^m_{g,\mathrm{v}}(M)$
			there exists a sequence
			$(u_{\varepsilon})_{\varepsilon>0}\subset W^{1,2}_{\mathrm{v}}(M,\mathbb{R}^m)$
			that converges in $L^1(M,\mathbb{R}^m)$
			to $Z_\Omega$ and 
			\begin{equation}
				\label{eq:limsup-prop}
				\limsup_{\varepsilon \to 0}
				\mathcal{E}_{\varepsilon,\mathrm{v}}(u_\varepsilon)
				\le \mathcal{P}^\omega_{g}(\Omega).
			\end{equation}
	\end{itemize}
	Moreover, the functionals  
	$\mathcal{E}_{\varepsilon,\mathrm{v}}$
	are \emph{equicoercive}, that is, for any sequence 
	$(u_{\varepsilon})_{\varepsilon>0}\in W^{1,2}_{\mathrm{v}}(M,\mathbb{R}^m)$
	such that
    \begin{equation}\label{prop:equicoerciveness}
        \limsup_{\varepsilon\to 0}
		\mathcal{E}_{\varepsilon,\mathrm{v}}(u_\varepsilon)
		< \infty,
    \end{equation}
	there exists $\Omega \in \mathcal{C}^m_{g,\mathrm{v}}(M)$
	such that $(u_{\varepsilon})_{\varepsilon>0}$
	converges in $L^1(M,\mathbb{R}^m)$ to $Z_\Omega$,
	up to subsequences.
\end{proposition}
\begin{proof}
Using compactness of $M$, there exists a finite subcover of 
    \[
        \big(B(p,\inj_M)\big)_{p \in M},
    \]
    that we denote by $\{B(p_1,\inj_M),\dots, B(p_D,\inj_M)\}$. Let $\{\rho_j\}_{j = 1}^D\subset C^\infty(M)$ be a subordinate partition of unit of this finite subcover, hence $\supp \rho_j \subset B(p_j,\inj_M)$ and $0\le \rho_j\le 1$ for any $j=1,\dots, D$ and $\sum_{j=1}^D\rho_j(p)=1$ for all $p\in M$. From this, it follows that the exponential map ${\rm exp}_{p_j}: B(0,\inj_M) \subset \mathbb{R}^N \to B(p_j,\inj_M)$ for each $j=1,\dots, D$ is a diffeomorphism. For any $\Omega=(\Omega^1,\dots\Omega^m) \in \mathcal{C}^m_{g,\mathrm{v}}(M)$ and each $j\in 1, \dots, D$, we define 
    \[
    \widehat{\Omega}^i_j:={\rm exp}_{p_j}^{-1}(\Omega^i
    \cap B(p_j,\inj_M)), \qquad \forall i=1,\dots, m,
    \]
    so that
    \[
    \widehat{\Omega}_j=(\widehat{\Omega}^1_j,\dots\widehat{\Omega}^m_j) \in \mathcal{C}^m_{\delta}(B(0,\inj_M)).
    \]
    Moreover, we set
    \[    
    \hat{Z}_{\widehat{\Omega}_j}
    \coloneqq
    \sum_{i=1}^m\mathbf{z}_i\chi_{\widehat{\Omega}^i_j}\in L^1(B(0,\inj_M),\mathbb R^m).
    \]
    By relying on this partition of the unity, the proof of the \emph{liminf property} and of the \emph{limsup property} follows from an application of~\cite[Theorem 2.5]{MR1051228}.

    Indeed, let $(u_{\varepsilon})_{\varepsilon>0}\subset W^{1,2}_\mathrm{v}(M,\mathbb{R}^m)$ converging in $L^1$--norm to $Z_\Omega$. For each $j=1,\dots, D$ and $\varepsilon>0$, we set $\hat{u}_{\varepsilon,j}\in W^{1,2}_{\hat{\mathrm{v}}_j}(B(p_{j},\inj_M),\mathbb R^m)$ given by 
    \[
        \hat{u}_{\varepsilon,j}(x)\coloneqq (\rho_j u_\varepsilon)({\rm exp}_{p_j}(x))=\rho_j({\rm exp}_{p_j}(x))u_\varepsilon({\rm exp}_{p_j}(x)), \qquad \forall x\in B(0,\inj_M),
   \]
   where 
   $\hat{\mathrm{v}}_j:=V_g(\hat{u}_{\varepsilon,j})$. We can apply \cite[Proof of (2.8)]{MR1051228} to conclude 
    \begin{equation*}
            \liminf_{\varepsilon \to 0}
            \mathcal{E}_{\varepsilon,\hat{\mathrm{v}}_j}(\hat{u}_{\varepsilon,j})
            \ge \mathcal{P}^\omega_{g}(\widehat{\Omega}_j), \qquad \forall j=1,\dots, D,
    \end{equation*}
    which, since $u_\varepsilon(p)=\sum_{j=1}^D \hat{u}_{\varepsilon,j}({\rm exp}^{-1}_{p_j}(p))$, implies
    \begin{equation*}
            \liminf_{\varepsilon \to 0}
            \mathcal{E}_{\varepsilon,\mathrm{v}}({u}_{\varepsilon})\geq \liminf_{\varepsilon \to 0}\left(\sum_{j=1}^D
            \mathcal{E}_{\varepsilon,\hat{\mathrm{v}}_j}(\hat{u}_{\varepsilon,j})\right)\geq \sum_{j=1}^D\left(\liminf_{\varepsilon \to 0}
            \mathcal{E}_{\varepsilon,\hat{\mathrm{v}}_j}(\hat{u}_{\varepsilon,j})\right)
            \ge \sum_{j=1}^D\mathcal{P}^\omega_{g}(\Omega_j)\ge \mathcal{P}^\omega_{g}(\Omega),
    \end{equation*}
    where used that 
    \begin{equation}\label{relationchambers1}
        Z_\Omega(p)=\sum_{j=1}^D\hat{Z}_{\widehat{\Omega}_j} \left({\rm exp}_{p_j}^{-1}(p)\right)=\sum_{j=1}^D\sum_{i=1}^m\mathbf{z}_i\chi_{\widehat{\Omega}^i_j}\left({\rm exp}_{p_j}^{-1}(p)\right), \quad \forall p\in M.
    \end{equation}
    This proves the liminf property.

    To prove the limsup property, we observe that since $\hat{Z}_{\widehat{\Omega}_j}\in L^1(B(0,\inj_M),\mathbb R^m)$, we can invoke \cite[Proof of (2.9)]{MR1051228} to conclude that for each $j=1,\dots, D$, one can find $(\hat{u}_{\varepsilon,j})_{\varepsilon>0}\subset W^{1,2}_{\hat{\mathrm{v}}_j}(B(0,\inj_M),\mathbb{R}^m)$ such that $\norm{\hat{u}_{\varepsilon,j} - \hat{Z}_{\widehat{\Omega}_j}}_{L^1(B(0,\inj_M),\mathbb{R}^m)}$ goes to zero and 
    \begin{equation*}
				\limsup_{\varepsilon \to 0}
				\mathcal{E}_{\varepsilon,\hat{\mathrm{v}}_j}(\hat{u}_{\varepsilon,j})
				\le \mathcal{P}^\omega_{g}(\widehat{\Omega}_j).
	\end{equation*}
    Hence, defining $u_{\varepsilon}\in W^{1,2}(M,\mathbb{R}^m)$ as 
    \begin{equation*}
        u_{\varepsilon}(p):=\sum_{j=1}^D\rho_j(p)\hat u_{\varepsilon,j}({\rm exp}^{-1}_{p_j}(p)),
    \end{equation*}
    and using \eqref{relationchambers1},
    it is not hard to verify that 
    $(u_{\varepsilon})_{\varepsilon>0}\subset W^{1,2}_\mathrm{v}(M,\mathbb{R}^m)$ 
    converges to $Z_\Omega$ in $L^1$--norm.
    Furthermore, one has 
    \begin{equation*}
            \limsup_{\varepsilon \to 0}
            \mathcal{E}_{\varepsilon,{\mathrm{v}}}({u}_{\varepsilon})
            \le\limsup_{\varepsilon \to 0}\left(\sum_{j=1}^D
            \mathcal{E}_{\varepsilon,\hat{\mathrm{v}}_j}(\hat{u}_{\varepsilon,j})\right)
            \le \sum_{j=1}^D\mathcal{P}^\omega_{g}(\widehat{\Omega}_j)\le \mathcal{P}^\omega_{g}({\Omega}),
    \end{equation*}
    which shows that \eqref{eq:limsup-prop} also holds, and proves the claim.

    The proof of equicoerciveness follows with minor modifications from \cite[Proposition 2.7]{MR3847750}. Specifically, by employing the same approach, one can apply \cite[Proposition 2.7]{MR3847750} in each open coordinate patch $B(p_j,\inj_M)\subset M$, and then extend the result globally across the manifold using a standard partition of unity argument.
\end{proof}

\begin{remark}
	\label{rem:recovery-sequences-continuity}
	For any $\Omega\in \mathcal{C}^m_{g,\mathrm{v}}(M)$,
	a sequence $(u_{\varepsilon})_{\varepsilon > 0} \subset W^{1,2}_{\mathrm{v}}(M,\mathbb{R}^m)$
	that converges to $Z_\Omega$ and such that~\eqref{eq:limsup-prop} holds is called~\emph{recovery sequence} for $\Omega$.
	By then applying the \emph{liminf property}, one obtains that actually
	the following equality holds for every recovery sequence:
	\[
		\lim_{\varepsilon \to 0}
		\mathcal{E}_{\varepsilon,\mathrm{v}}(u_{\varepsilon})
		= \mathcal{P}^{\omega}_g(\Omega).
	\]
	By the construction of the recovery sequences
	(see, e.g., \cite{MR1051228,MR4701348}),
	they can be chosen continuously
	with respect to the associated clusters.
	This means that, 
	if we have a cluster $\Omega_0 \in \mathcal{C}^m_{g,\mathrm{v}}(M)$
	and a sequence of clusters 
	$(\Omega_k)_{k\in\mathbb N} \subset \mathcal{C}^m_{g,\mathrm{v}}(M)$
	that converges to $\Omega_0$,  {\it i.e.},
	\[
		\lim_{k \to \infty} \dist(\Omega_0,\Omega_k)
		= \lim_{k \to \infty} \sum_{i = 1}^m
		|\Omega_0^i \triangle \Omega_k^i|
		= 0,
	\]
	then, we can choose, for each $\Omega_k\in \mathcal{C}^m_{g,\mathrm{v}}(M)$,
	$k \in \mathbb N$, a recovery sequence
	$(u_{k,\varepsilon})_{\varepsilon > 0} \subset W^{1,2}_{\mathrm{v}}(M,\mathbb{R}^m)$
	such that
	\[
		\lim_{k \to \infty}
		\norm{u_{0,\varepsilon} - u_{k,\varepsilon}}_{W^{1,2}(M,\mathbb{R}^m)}
		= 0.
	\]
	for any $\varepsilon > 0$.
\end{remark}

\begin{remark}
     By Proposition~\ref{prop:gamma-convergence},
     we notice that, for any $u_{\varepsilon}\in W^{1,2}_\mathrm{v}(M,\mathbb{R}^m)$ solution to~\eqref{eq:ACH-PDE},
     one can construct a $m$-cluster $\Omega_\varepsilon\in \mathcal{C}^m_{g}(M)$ such that $\Omega_\varepsilon=(\Omega_\varepsilon^1,\dots,\Omega_\varepsilon^m)$ where
     \begin{equation*}
    	\Omega_\varepsilon^i\coloneqq(\varphi_i\circ u_\varepsilon)^{-1}({\mathbf{z}_i}), \quad \forall i=1,\dots, m,
    \end{equation*}
    with $\varphi_i: \mathbb{R}^m\rightarrow \mathbb{R}$ being defined by $\varphi_i(z)=\ud_W(z,\mathbf{z}_i)$ for $i=1,\dots,m$, and $\ud_W$ denotes the distance induced by the Riemannian metric provided by $W$ (see \eqref{eq:def-omega-ij}). 
    Associated to this $m$-cluster, we can also construct a function $Z_{\Omega_\varepsilon}\in L^1(M,\mathbb{R}^m)$ given by $Z_{\Omega_\varepsilon}=\sum_{i=1}^m\mathbf{z}_i\chi_{\Omega_\varepsilon^i}$.
    In this fashion, the $L^1$ convergence of 
    $(u_\varepsilon)_\varepsilon$ to
    $Z_\Omega$ can be interpreted as the interface (flat) convergence of the level sets of the Allen--Cahn energy, that is
    \begin{equation*}
        \lim_{\varepsilon\to0}\norm{u_\varepsilon - Z_\Omega}_{L^1(M,\mathbb{R}^m)}=\lim_{\varepsilon\to0}\dist(\Omega_\varepsilon,\Omega)= \sum_{i = 1}^m\lim_{\varepsilon\to0}|\Omega_\varepsilon^i \triangle \Omega^i|=0.
    \end{equation*}
\end{remark}

\section{Outline for the Photography Method}
\label{sec:setting}

In this section, we provide some preliminary definitions and
the necessary setup to apply the photography method,
specifically explaining how we employ
Theorem~\ref{theorem:abstract-photography}
in our setting.
The formal constructions of the photography and barycenter maps
are postponed to Section~\ref{sec:photographymap}
and Section~\ref{sec:barycenter}, respectively.
This is intended as an informal presentation to outline 
the main ideas.

\subsection*{Photography Map}
\label{subsec:photography}

The main idea of constructing the photography map is similar to the one employed in~\cite{MR4073210,MR4701348},
and can be described as follows. 
For any point on the manifold,
we consider a weighted multi-isoperimetric region
$\Omega\in \mathcal{C}^m_{\delta,\mathrm{v}}(\mathbb R^N)$
in the tangent space of the manifold at that point,
for the given volume $\mathrm{v} \in \mathbb{R}^m_{>0}$.
We then project this isoperimetric region onto the manifold
using the exponential map,
resulting in a cluster $\widetilde{\Omega} \in \mathcal{C}^m_g(M)$.
Finally, we obtain a function in $W^{1,2}(M,\mathbb{R}^m)$
by exploiting the recovery sequence construction
ensured by Proposition~\ref{prop:gamma-convergence}
for the function $Z_{\widetilde{\Omega}}\in L^1(M,\mathbb{R}^m)$.
However, some technicalities need to be solved to
successfully apply this idea in our setting.
To better explain the main difficulties
let us introduce the following maps,
even if their formal definitions will be given 
in the subsequent sections:
\begin{itemize}
	\item[{\rm (i)}] $J\colon \mathbb{R}^m_{>0} \to \mathcal{C}^m_\delta(\mathbb{R}^N)$ is the map that associates to each vectorial volume
		a multi-isoperimetric region in the $N$-dimensional euclidean space;
	\item[{\rm (ii)}] $\Pi\colon M \times \mathcal{C}^m_\delta(\mathbb{R}^N)
		\to \mathcal{C}^m_{g}(M)$
		is the projection map of the $m$-cluster near the 
		selected point of the manifold;
	\item[{\rm (iii)}] $R_\varepsilon\colon \mathcal{C}^m_{g}(M)
		\to W^{1,2}(M,\mathbb{R}^m)$
		is the map that gives the element of the recovery sequence
		which corresponds to $\varepsilon$.
\end{itemize}
By using this notation,
the idea is to define the photography map
$\varphi_{\varepsilon,\mathrm{v}}\colon M \to W^{1,2}_{\mathrm{v}}(M,\mathbb{R}^m)$
as follows:
\[
	\varphi_{\varepsilon,\mathrm{v}}(p)
	= R_\varepsilon\left(\Pi(p,J(\mathrm{v})\right).
\]
This construction is illustrated in Figure~\ref{fig:photography}.

However, there are two technical and interesting challenges
to overcome in applying the above scheme to our setting, challenges that were not present
in~\cite{MR4073210} and in~\cite{MR4701348},
where the same problem for $m = 1$ and $m = 2$ was considered, respectively.

The first one is that the projection map $\Pi$ could not be well defined:
in fact, if the cluster $J(\mathrm{v})$ has large diameter,
the exponential function can fail to give a ``nice'' copy
of the multi-isoperimetric region in the manifold.
When $m = 1$ or $m = 2$, the isoperimetric regions on the tangent
space are fully characterized,
as they are the spheres and the so-called
double bubbles.
As a consequence, if $\mathrm{v}$ is sufficiently small,
they have a diameter strictly lower than the injectivity radius
of the manifold, so one can ensure that the 
exponential map can be used to project them on the manifold.
In our case, this characterization is not available:
a priori, a multi-isoperimetric region on $\mathbb{R}^N$ 
can not have a small diameter,
even if $\mathrm{v}$ is small,
so that the projection on the manifold could not be possible.
We also notice that when we perform the projection,
we need to use the parallelizability of the manifold
to ensure that the projection map is continuous on the whole manifold.

The second problem is somehow more subtle:
provided that the projection on the manifold is possible,
the $m$--cluster $\widetilde{\Omega} \in \mathcal{C}_g^m(M)$
we obtain has a vectorial volume 
$\tilde{\mathrm{v}}$ which is slightly different 
from the desired one,
due to the distortion given by the curvature of the manifold.
If $m=1,2$, this is not a big deal since the characterization 
of the isoperimetric regions on $\mathbb{R}^N$
ensures the \emph{continuity} of them
with respect to the volume parameter
(when one uses the flat norm on the $m$-clusters),
so one can change the original volume of the isoperimetric-region
$\Omega \in \mathcal{C}^m_\delta(\mathbb{R}^N)$
in such a way that its projection has the desired volume $\mathrm{v}$,
and this choice is also continuous with respect to 
the points of the manifold where we are performing this construction,
due to the continuity of the curvature tensor.

\begin{figure}[h]
	\centering
	\begin{picture}(468, 184) 
		\put(-10,0){\includegraphics[width=\textwidth]{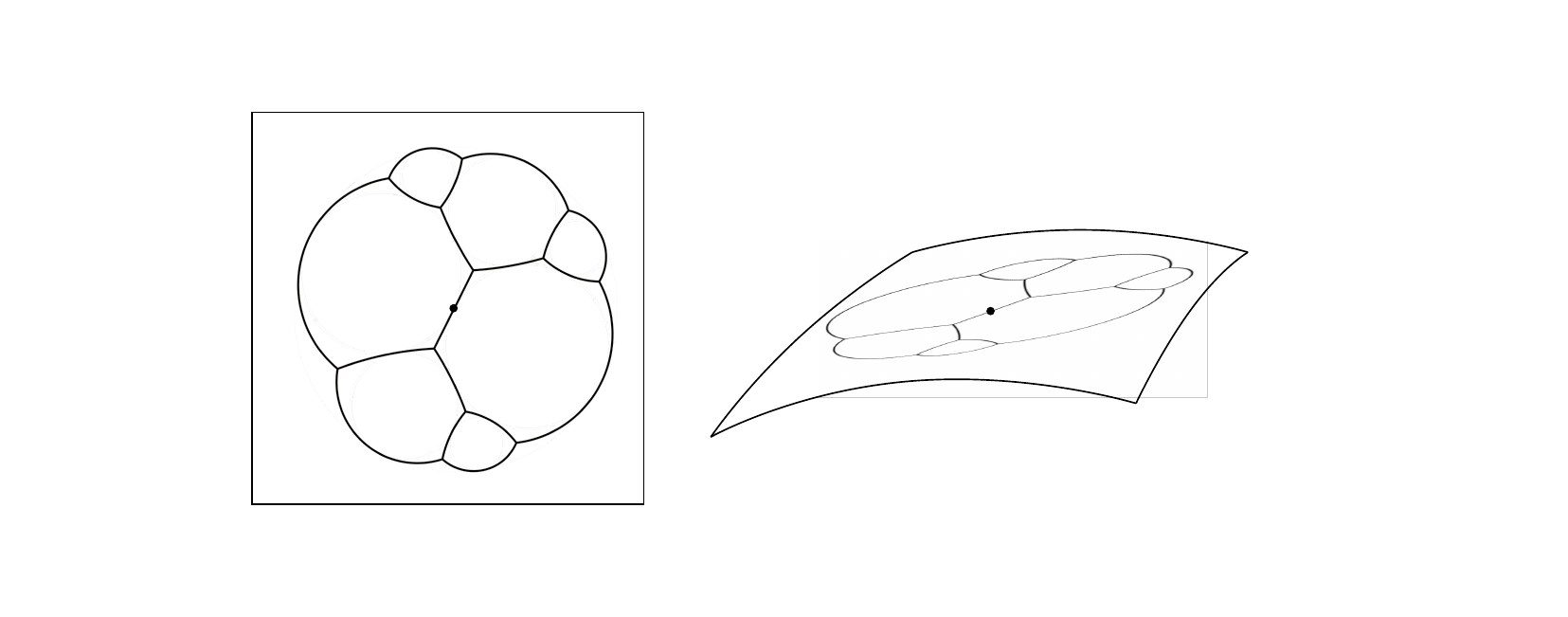}}
		\put(28,90){$\mathrm{v} \longrightarrow$}
		\put(41,100){$J$}
		\put(73,120){$\Omega$}
		\put(125,85){$0$}
		\put(170,150){$\mathbb{R}^N$}
		\put(190,100){$\Pi(p,\cdot)$}
		\put(190,90){${\longrightarrow}$}
		\put(283,82){$p$}
		\put(300,70){$\widetilde{\Omega}$}
		\put(324,60){$M$}
		\put(363,100){$R_\varepsilon$}
		\put(360,90){$\longrightarrow W^{1,2}_{\mathrm{v}}(M,\mathbb{R}^m)$}
	\end{picture}
	\caption{The main idea for the construction of the photography map.
	In this figure, we have $N = 2$ and $m = 7$.}
	\label{fig:photography}
\end{figure}
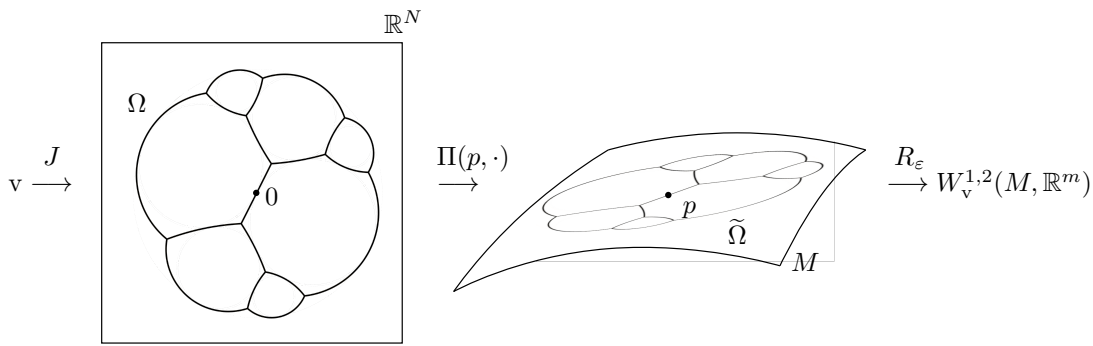

When $m \ge 3$, the continuity of the multi-isoperimetric regions
with respect to the prescribed vectorial volume $\mathrm{v} \in \mathbb{R}^m$
is not guaranteed;
in other words, the above map $J$ can not be continuous.
As a result, the previous construction does not work
since we can end with two different $m$--chambers on
$\mathcal{C}^m_{g,\mathrm{v}}(M)$ for two arbitrary
close points of the manifolds.
To address this issue, we proceed as follows.
The key observation is that
\emph{there is no need to 
choose multi-isoperimetric regions on the tangent space}.
What we really need is to associate with every point of the manifold
an $m$-cluster
(and then a function in $W^{1,2}_{\mathrm{v}}(M,\mathbb{R}^m)$
associated with it)
whose weighted multi-perimeter is near,
in a precise sense that will be later explained,
to the minimum $\mathcal{I}^\omega_g(\mathrm{v})$.
Indeed, even if one considers multi-isoperimetric regions,
the projection onto the manifold leads to 
clusters whose perimeter could be slightly higher than the 
minimum, always because of the curvature of the manifold.
The crucial step, even when $m=1,2$,
is to estimate this increase in the perimeter,
so that the perimeter of the projected cluster has
to be less than some quantity of the following kind:
\begin{equation*}
	\mathcal{I}^\omega_g(\mathrm{v}) + \mathrm{o}(\mathcal{I}^\omega_g(\mathrm{v}))
	\quad \text{as } \quad \mathrm{v} \to 0.
\end{equation*}
Results of this kind can be found in~\cite[Proposition 3.1]{MR4644903} for $m = 1$
and \cite[Lemma 5.8]{MR4701348} for $m = 2$.
Therefore, we can change the definition of 
the map $J\colon \mathbb{R}^m_{>0} \to \mathcal{C}^m_\delta(\mathbb{R}^N)$,
as long as we can control the perimeter of the  
obtained cluster.
Moreover, there is no need to have a continuous function on 
the whole $\mathbb{R}^m_{>0}$, 
but just in a neighborhood of
the desired vectorial volume $\mathrm{v}$,
provided that this neighborhood is sufficiently large
to compensate for the distortion of the projection map.
Slightly more formally, for every $\mathrm{v} \in \mathbb{R}^m_{>0}$
we need a neighborhood of it,
say $U_{\mathrm{v}}\subset \mathbb{R}^m_{>0}$,
and a continuous function
$J_{\mathrm{v}}\colon U_\mathrm{v} \to \mathcal{C}^m_\delta(\mathbb{R}^N)$
such that, for every $p \in M$,
there exists $\mathrm{v}_p \in U_{\mathrm{v}}$
such that
\[
	\mathcal{V}_g\big(\Pi(p,J_{\mathrm{v}}(\mathrm{v}_p))\big) = \mathrm{v},
\]
namely,
$\Pi(p,J_{\mathrm{v}}(\mathrm{v}_p)) \in \mathcal{C}^m_{g,\mathrm{v}}(M)$,
and 
\begin{equation}
	\label{eq:estimate_perimeter_projection}
	\mathcal{P}^\omega_g\big(\Pi(p,J_{\mathrm{v}}(\mathrm{v}_p))\big)
	\le \mathcal{I}^\omega_g(\mathrm{v}) + \tau(\mathrm{v}),
\end{equation}
where $\tau\colon \mathbb{R}^m_{>0} \to \mathbb{R}$
is an infinitesimal function with respect to $\mathcal{I}^\omega_g$,
namely
\begin{equation}
	\label{eq:tau-infinitesimal}
	\lim_{\mathrm{v}\to 0}
	\frac{\tau(\mathrm{v})}{\mathcal{I}^{\omega}_g(\mathrm{v})} = 0.
\end{equation}

The existence of such a map
$J_\mathrm{v}\colon U_{\mathrm{v}} \to \mathcal{C}_\delta^m(\mathbb{R}^N)$
can be ensured by a ``fixing volume variations'' lemma,
which in our setting is formally given by Lemma~\ref{lem:fixing-volumes}.
Its construction relies on the application of
the inverse function theorem,
and, as a result,
\emph{we are unable to provide an estimate}
for the size of $U_{\mathrm{v}}$.
This estimate is, however, crucial to ensure that 
it can compensate for the distortion of the projection map,
as previously explained.
To address this important technical challenge,
we rely on the scaling invariant property of the 
multi-perimeter functional;
namely, if $\Omega \in C_{\delta,\mathrm{v}}^m(\mathbb{R}^N)$,
we can ``resize'' it by rescaling it by a factor of $\alpha > 0$,
obtaining a cluster $\alpha\Omega$ whose volume is
$\alpha\mathrm{v}$.
By a standard computation, it follows that 
\begin{equation}
	\label{eq:scaling-law}
	\mathcal{P}_{\delta}^{\omega}(\alpha\Omega)
	= \alpha^{\frac{N-1}{N}}\mathcal{P}_{\delta}^{\omega}(\Omega).
\end{equation}
As a consequence, by a rescaling procedure, we 
can define a map $J_{\alpha\mathrm{v}}$
starting from $J_{\mathrm{v}}$,
and we are able to control 
the size of the rescaled set $U_{\alpha\mathrm{v}}$
where this map is defined.
More precisely, let us consider an open nonempty ball
$B(\mathrm{v},\zeta)\subset U_{\mathrm{v}} \subset \mathbb{R}^m$,
where we stress the lack of any quantitative estimate on 
the value of $\zeta > 0$.
However, we can use the restriction of $J_{\mathrm{v}}$
on $B(\mathrm{v},\zeta)$ and,
for any $\alpha > 0$,
we can define the map $J_{\alpha \mathrm{v}}$
on the open ball
\[
	B(\alpha\mathrm{v}, \alpha^{1/N}\zeta),
\]
as follows:
\[
	J_{\alpha,\mathrm{v}}(\tilde{\mathrm{v}})
	= \alpha J_{\mathrm{v}}\big(\tilde{\mathrm{v}}/\alpha\big).
\]
Thus, even without specifying $\zeta$,
by taking the limit $\alpha \to 0$,
\emph{we can estimate how the size of the neighborhood around $\alpha\mathrm{v}$},
where $J_{\alpha\mathrm{v}}$ is defined,
\emph{approaches to zero}.
More importantly, we can demonstrate that if $\alpha$ is sufficiently small,
\emph{this neighborhood is large enough to compensate
for the distortion caused by the manifold's curvature}
through the projection map.
Additionally, we can control the multi-perimeter of the clusters provided by
$J_{\alpha\mathrm{v}}$ by obtaining an estimate analogous
to~\eqref{eq:estimate_perimeter_projection}.

The price to pay is that,
by sending to zero the vectorial volume in this way,
the ratio between the volumes of the chambers remains
constant, and this is directly reflected in the statement
of Theorem~\ref{theorem:main}, where 
the new quantifier $\alpha^*> 0$ appears with respect
to the analogous theorems given on this subject for the case of one or two chambers (see, e.g., ~\cite{MR4073210,MR4701348}).

Finally, if $\alpha$ is sufficiently small,
we can define the photography map as follows:
\[
	\varphi_{\varepsilon,\alpha\mathrm{v}}(p)
	=
	R_{\varepsilon}\big(\Pi(p,J_{\alpha\mathrm{v}}(\alpha\mathrm{v}_p))\big),
\]
and~\ref{item:s1} can be performed by proving the following result.

\begin{proposition}
	\label{prop:photography}
	For every $\mathrm{v} \in \mathbb{R}^m_{>0}$
	there exists $\alpha^*_1 = \alpha^*_1(M,g,\mathrm{v}) > 0 $
	and a function $\tau_{\mathrm{v}}\colon (0,\alpha^*_1) \to \mathbb{R}$
	such that
	\begin{equation}
		\label{eq:tau-infinitesimal-alpha}
		\lim_{\alpha \to 0}
		\frac{\tau_{\mathrm{v}}(\alpha)}{\mathcal{I}^{\omega}_\delta(\alpha\mathrm{v})} = 0,
	\end{equation}
	and 
	for every $\alpha \in (0,\alpha^*_1)$
	there exists an $\varepsilon^*_1 = \varepsilon^*_1(\alpha,\mathrm{v})>0$
	such that for every $\varepsilon \in (0,\varepsilon^*_1)$
	the photography map 
	$ \varphi_{\varepsilon,\alpha\mathrm{v}}\colon 
	M\to W_{\mathrm{v}}^{1,2}(M,\mathbb{R}^m)$
	is continuous and
	\[
		\varphi_{\varepsilon,\alpha\mathrm{v}}(p)
		\in
		\mathcal{E}^{\sigma_{\mathrm{v}}(\alpha)}_{\varepsilon,\alpha\mathrm{v}},
		\qquad \forall p \in M,
	\]
	where
	$\sigma_\mathrm{v}(\alpha)
	\coloneqq \mathcal{I}^{\omega}_\delta(\alpha\mathrm{v})
	+ \tau_{\mathrm{v}}(\alpha)$.
\end{proposition}

\subsection*{Barycenter Map}
\label{subsec:barycenter}
By Proposition~\ref{prop:photography},
we obtain that if $0<\alpha,\varepsilon\ll1$
are sufficiently small,
the image of the photography map $\varphi_{\varepsilon,\alpha\mathrm{v}}$
is a class of functions that 
have energy less than or equal to $\sigma_\mathrm{v}(\alpha)$,
where the latter satisfies the following equation:
\[
	\lim_{\alpha\to 0}
	\frac{\sigma_\mathrm{v}(\alpha)}{\mathcal{I}^\omega_\delta(\alpha\mathrm{v})} = 1.
\]
This last estimate is crucial to 
define the barycenter map,
as it ensures that any function in 
$\mathcal{E}^{\sigma_\mathrm{v}(\alpha)}_{\varepsilon,\alpha\mathrm{v}} $
should have almost all its mass in a
small geodesic ball,
provided $\alpha$ and $\varepsilon$ sufficiently small.
As a consequence, we will be able to construct a continuous map
$\beta\colon \mathcal{E}^{\sigma_\mathrm{v}(\alpha)}_{\varepsilon,\alpha\mathrm{v}} \to M$
utilizing the ``barycenter'' of these functions.
To achieve such a concentration result
for the functions in the sublevels 
$\mathcal{E}^{\sigma_\mathrm{v}(\alpha)}_{\varepsilon,\alpha\mathrm{v}}$,
we first prove an analogous result for the 
``almost isoperimetric'' clusters,
and then we exploit once again the $\Gamma$--convergence,
relying on the equicoerciveness of the functionals
$(\mathcal{E}_{\varepsilon,\alpha\mathrm{v}})_{\varepsilon >0}$.
To better present our construction,
let us define almost isoperimetric clusters as follows.
\begin{definition}
	\label{def:almost-isoperimetric-clusters}
	A sequence $(\Omega_{k})_{k\in \mathbb{N}}\subset \mathcal{C}^m_g(M)$
	is called \emph{almost isoperimetric} if the following equality holds
	\begin{equation}
		\label{eq:def-almost-isoperimetric-clusters}
		\lim_{k \to \infty}
		\frac{\mathcal{P}^{\omega}_g(\Omega_{k})}
		{\mathcal{I}^{\omega}_g(\mathrm{v}_{k})} = 1,
	\end{equation}
	where $\mathrm{v}_k\in \mathbb{R}^m_{>0}$ denotes
	the vectorial volume of $\Omega_k$, for all $k \in \mathbb{N}$.
\end{definition}

It is not difficult to see that if a sequence 
$(\mathrm{v}_k)_{k\in \mathbb{N}} \subset \mathbb{R}^m_{>0}$
is such that $\norm{\mathrm{v}_{k}}=\mathrm{o}_k(1)$,
then also $\mathcal{I}^{\omega}_g(\mathrm{v}_{k})$
is infinitesimal
(just consider a sequence $(\Omega_k)_{k \in \mathbb{N}}\subset \mathcal{C}^m_{g,\mathrm{v}_k}(M)$ 
given by disjoint geodesic balls
whose volumes are equal to the components of $\mathrm{v}_k$
for every $k\in\mathbb{N}$).
As a consequence, if $(\Omega_{k})_{k\in \mathbb{N}}\subset \mathcal{C}^m_g(M)$
is an almost isoperimetric sequence of clusters
such that $\mathcal{V}_{g}(\Omega_{k})=\mathrm{o}_k(1)$,
then~\eqref{eq:def-almost-isoperimetric-clusters} implies that
\begin{equation}
	\label{eq:almost-isoperimetric-first-term}
	\mathcal{P}^{\omega}_g(\Omega_k)
	= \mathcal{I}^{\omega}_g(\mathcal{V}_g(\Omega_k))
	+ \mathrm{o}_k(\mathcal{I}^{\omega}_g(\mathcal{V}_g(\Omega_k))).
\end{equation}
With the above definition,
we are ready to present a concentration result
for the almost isoperimetric clusters,
whose proof can be found in~\cite[Proposition~12]{MR4701348}.

\begin{proposition}
	\label{prop:concentration-almost-minimizers}	
	There exists $\mu > 0$ such that the following 
	property holds.
	For any infinitesimal sequence
	$(\alpha_k)_{k \in \mathbb{N}} \subset \mathbb{R}_{>0}$,
	if $(\Omega_k)_{k \in \mathbb{N}} \subset \mathcal{C}^m_g(M)$
	is an almost isoperimetric sequence
	$($see Definition~\ref{def:almost-isoperimetric-clusters}$)$
	such that $\mathcal{V}_g(\Omega_k)= \alpha_k\mathrm{v}$
	for any $k \in \mathbb{N}$,
	there exists
	a sequence of points $(p_k)_{k \in \mathbb{N}} \subset M$ 
	such that the sequence of clusters
	$(\widetilde{\Omega}_k)_{k \in \mathbb{N}} \subset \mathcal{C}^m_g(M)$
	defined as
	\[
		\widetilde{\Omega}_k \coloneqq
		{\Omega}_k \cap {B}_g(p_k, \mu(\alpha_k\norm{\mathrm{v}})^{1/N})
	\]
	satisfies the following properties:
	\begin{itemize}
		\item[{\rm (i)}] the symmetric differences between 
			$\Omega_k$ and $\widetilde{\Omega}_k$
			are small,  {\it i.e.},
			\[
				\lim_{k \to \infty}
				\frac{\norm{\mathcal{V}_g(\Omega_k \triangle \widetilde{\Omega}_k)}}
				{\norm{\mathcal{V}_g(\Omega_k)}} = 0;
			\]
		\item[{\rm (ii)}]
			their vectorial volumes approach those of $\Omega_k$,
			{\it i.e.},
			\[
				\lim_{k \to \infty}
				\frac{\mathcal{H}^{N}_g(\widetilde{\Omega}_k^i)}
				{\mathcal{H}^{N}_g(\Omega_k^i)}
				= 1,
				\qquad \forall i = 1,\dots,m;
			\]
		\item[{\rm (iii)}]
			their weighted multi-perimeters approach those of $\Omega_k$,
			{\it i.e.},
			\[
				\lim_{k \to \infty}
				\frac{\mathcal{P}^\omega_g(\widetilde{\Omega}_k)}{
				\mathcal{P}^\omega_g(\Omega_k)} = 1.
			\]
	\end{itemize}
	By construction, $\widetilde{\Omega}_k$ are contained within
	geodesic balls with diameters tending to zero,  {\it i.e.},
	\[
		\lim_{k \to \infty} \diam(\widetilde{\Omega}_k) = 0.
	\]
\end{proposition}
\begin{proof}
	The proof follows the same structure as that of~\cite[Proposition 12]{MR4701348},
	where concentration results for almost minimizing clusters are established in the case of two chambers ($m=2$ in our notation).
	Upon careful analysis of that proof, we observe that the only step where the characterization of the minimizers is used is in~\cite[(3.3), page 25]{MR4701348}. However, this characterization is employed solely to use the homogeneity of the multi-isoperimetric profile function,which is noting but~\eqref{eq:scaling-law}, implying that the proof holds for any number of chambers.
\end{proof}

Thanks to the $\Gamma$--convergence of the vectorial AC energy to the multi-isoperimetric profile, 
we can translate the concentration result provided by
Proposition~\ref{prop:concentration-almost-minimizers}
to the functions in
$\mathcal{E}_{\varepsilon,\alpha\mathrm{v}}^{\sigma_{\mathrm{v}}(\alpha)}$,
always under the condition that
both $\alpha$ and $\varepsilon$ are sufficiently small.
More formally, we have the following result.
\begin{proposition}
	\label{prop:almost-min-concentration}
	There exists a constant $\mu > 0$
	such that the following property holds.
	For every $\theta \in (0,1)$ and $\mathrm{v}\in\mathbb{R}_{>0}^m$,
	there exists $\alpha_2 = \alpha_2(\theta,\mathrm{v}) > 0$
	such that for every $\alpha \in (0,\alpha_2)$
	there exists $\varepsilon_2 = \varepsilon_2(\theta,\alpha) > 0$
	such that for any
        $\varepsilon \in (0,\varepsilon_2)$
        and 
	$u \in 
	\mathcal{E}^{\sigma_{\mathrm{v}}(\alpha)}_{\varepsilon,\alpha\mathrm{v}}$
	there exists $p\in M$
	such that
	\begin{equation}\label{concentration}
		\int_{M\setminus B_g(p,\mu(\alpha\norm{\mathrm{v}})^{1/N})}
		\norm{u} \de v_g
		\le \theta\, \alpha\norm{\mathrm{v}}.
	\end{equation}
\end{proposition}

Loosely speaking, Proposition~\ref{prop:almost-min-concentration}
guarantees that if $\alpha$ and $\varepsilon$ are sufficiently small,
any function in
$\mathcal{E}_{\epsilon,\alpha\mathrm{v}}^{\sigma_{\mathrm{v}}(\alpha)}$
has almost all its mass concentrated within a small geodesic ball.
By neglecting the contribution outside this ball,
we can then define the barycenter of such a function.
This is the main idea behind the definition of
the \emph{intrinsic} barycenter,
as discussed, for example, in~\cite{MR442975,MR3912791},
and used also in~\cite{MR4761862}.
We highlight that, in the previous papers about the multiplicity 
results of the (AC) system (e.g.\cite{MR4644903,MR4701348}),
an \emph{extrinsic} barycenter was 
defined by relying on the Nash embedding theorem,
considering the manifold as a subset of an Euclidean space
(see, e.g.,~\cite{MR4701348,MR4644903,MR4396580,arXiv:2401.17847}).

In order to define the intrinsic barycenter map,
let us introduce some further notation.
We adapt the approach of J. Petean~\cite{MR3912791} (see also \cite{MR442975}) for the case of $m$-maps.
Since $(M,g)$ is closed, one has
\begin{equation}
	\label{r_0}
	r_0\coloneqq
	\sup\left\{\bar{r}>0 : \text{$B_g(p,r)$ is strongly convex for all $p \in M$ and $r\in(0,\bar{r}]$} \right\}>0.
\end{equation} 
For $r\in(0,r_0)$, let us set
\begin{equation*}
	L^{1,r}(M,\mathbb{R}^m)\coloneqq
	\big\{u\in L^{1}(M,\mathbb{R}^m)\setminus \{0\}:
	\supp \norm{u}\subset {B}_g(p,r) \text{ for some } p\in M \big\}.
\end{equation*}
For any $u \in L^{1,r}(M,\mathbb{R}^m)$,
we define $P_u\colon M \to \mathbb{R}$
as follows:
\[
	P_{u}(p)\coloneqq\int_{M} \mathrm{dist}_g^2(p, x) \norm{u(x)} \ud v_{g}.
\]
Roughly speaking, $P_u$ is a cost function that 
penalizes the points away from the support of the 
function $u$.
Hence, we can define
the \emph{intrinsic barycenter map}
$\beta: L^{1,r}(M,\mathbb{R}^m) \rightarrow M$
as the unique minimizer of $P_u$,
formally we have
\begin{equation}
	\label{eq:def-intrinsic-barycenter}
	\beta(u)
	\coloneqq 
	\big\{ p \in M: P_u(p) \le P_u(q), \, \forall q \in M\big\}.
\end{equation}
We remark that such a minimizer is unique;
indeed, if $r\in(0,r_0)$ and $u \in L^{1,r}(M,\mathbb{R}^m)$,
the function $P_u$ is strictly convex in a small ball
and the minimum cannot be achieved outside such a ball
(see \cite[Section 1]{MR442975} for more details).

In order to define the barycenter
of functions whose \emph{almost all }support is concentrated
within a small ball, we proceed as follows. 
For any function $u \in L^{1}(M,\mathbb{R}^m)\setminus\{0\}$ and $r>0$,
we define the $(u,r)$-concentration function
$C_{u,r}\colon M \to \mathbb{R}$
as follows:
\[
	C_{u,r}(p) \coloneqq
	\frac{\int_{B_g(p, r)} \norm{u} \ud v_{g}}{\int_{M} \norm{u} \ud v_{g}},
\]
and then $r$-concentration functional 
$C_r\colon L^{1}(M,\mathbb{R}^m) \to \mathbb{R}$
as
$C_r(u) \coloneqq \max_{p\in M} C_{u,r}(p)$.
Consequently, for any $\eta\in(0,1)$, we  can define
\begin{equation*}
	L_{r, \eta}^{1}(M,\mathbb{R}^m)
	\coloneqq \left\{
		u \in L^{1}(M,\mathbb{R}^m)\setminus\{0\} :
	C_{r}(u)>\eta \right\}.
\end{equation*}

Let us restrict $\eta \in(\frac{1}{2},1)$ and consider the piecewise linear continuous function (cut-off) $\phi_{\eta}\colon \mathbb{R} \to [0,1]$ such that $\phi_{\eta}(t)=0$ if $t\in(-\infty,1-\eta]$, $\phi_{\eta}(t)\equiv1$ if $t\in[\eta,\infty)$ and it is an affine function and increasing in $[1-\eta, \eta]$.
By using this construction, 
for any $u\in L^1_{r,\eta}(M,\mathbb{R}^m)$,
we define the modified function
\begin{equation}
    \label{eq:def-Phi-concentration}
    \Phi_{r, \eta}(u)(x)\coloneqq
    \phi_{\eta}(C_{u,r}(x)) u(x),
\end{equation}
which, loosely speaking, cancels out
the small contributions of $u$
outside the main ball where its support is concentrated.
It is easy to see that,
for any $r\in(0,\frac{r_0}{2})$
and $\eta\in(1/2,1)$,
we have
\[
	\Phi_{r, \eta}(u) \in L^{1,2r}(M,\mathbb{R}^m),
	\qquad \forall u \in L_{r, \eta}^{1}(M,\mathbb{R}^m).
\] 
Indeed, 
if $u \in L_{r,\eta}^1(M,\mathbb{R}^m)$
and $p_0\in M$ is any point where
$C_{u,r}$ has a maximum,
we have $C_{u,r}(p_0) \ge \eta > 1/2$,
Hence, if $\mathrm{dist}_g(x, p_0) > 2r$,
then $C_{u, r}(x) \le 1 - \eta$
and, by definition of $\phi_\eta$,
we find $\phi_{\eta}\left(C_{u, r}(x)\right)=0$, 
so $\Phi_{r,\eta}(u)(x) = 0$.
This means that 
$\supp\Phi_{r,\eta}(u)\subset {B}_{g}(p_0,2r)$,
so $\Phi_{r,\eta} \in L^{1,2r}(M,\mathbb{R}^m)$.
Moreover, $\Phi_{r,\eta}$
is continuous as composition of continuous maps.

As a consequence,
we can extend the notion of barycenter to those 
functions whose support is \emph{almost} all contained
in small balls by composing $\Phi_{r, \eta}$
with the map $\beta$ defined in~\eqref{eq:def-intrinsic-barycenter}.
Formally, this construction leads to the following definition.
\begin{definition}\label{def:barycenter}
	For any $r\in(0,\frac{r_0}{2})$,
	and any $\eta\in(1/2,1)$,
	we define the 
	\emph{intrinsic barycenter map}
	$\beta_{r, \eta}\colon L^1_{r,\eta}(M,\mathbb{R}^m)\to M$ as follows:
	\begin{equation}
		\label{eq:def-beta}
		\beta_{r, \eta}(u)(x)\coloneqq 
            \beta(\Phi_{r,\eta}(u)).
	\end{equation}
\end{definition}

Thanks to this last definition, 
we can finally define a barycenter map
on
$\mathcal{E}_{\varepsilon,\alpha\mathrm{v}}^{\sigma_{\mathrm{v}}(\alpha)}$
by proving that this sublevel is a subset
of a suitable $L_{r,\eta}^1(M,\mathbb{R}^m)$,
provided $\varepsilon$ and $\alpha$ sufficiently small.
More formally, we will need to prove the following
result.

\begin{proposition}
	\label{prop:barycenter}
	For any $r \in (0,\frac{r_0}{2})$,
	for any 
	$\eta\in(1/2,1)$,
	and any $\mathrm{v}\in\mathbb{R}_{> 0}^m$,
	there exists $\alpha^*_2 = \alpha^*_2(\mathrm{v},r,\eta) > 0$
	such that for every $\alpha \in (0,\alpha^*_2)$
	there exists
	$\varepsilon^*_2 = \varepsilon^*_2(\alpha,\mathrm{v}) > 0$
	such that
	$\mathcal{E}^{\sigma_{\mathrm{v}}(\alpha)}_{\varepsilon,\alpha\mathrm{v}}\subset L^1_{r,\eta}(M,\mathbb{R}^m)$
	for any $\varepsilon\in(0,\varepsilon^*_2)$.
	Therefore, the intrinsic barycenter map
	$\beta_{r, \eta}\colon
	\mathcal{E}^{\sigma_{\mathrm{v}}(\alpha)}_{\varepsilon,\alpha\mathrm{v}} \to M$
        is well defined and continuous.
\end{proposition}

\section{Photography Map}
\label{sec:photographymap}

In this section, we provide the complete details for constructing the photography map, thus formally proving Proposition~\ref{prop:photography}.
From now on, let us fix a vectorial volume $\mathrm{v} \in \mathbb{R}^m_{>0}$. While the construction presented below depends on this specific volume, the procedure is general and can also be applied to any other vectorial volume.

Following the framework established in Section~\ref{subsec:photography}, let us consider a bounded isoperimetric $\omega$-weighted $m$-cluster in $\mathbb{R}^N$ with the given vectorial volume $\mathrm{v}$, denoted by $\Omega_{\mathrm{v}}$ (see Remark \ref{rem:bounded-isoperimetric-clusters} for the existence and boundedness of $\Omega_{\mathrm{v}}$).

The first step is to construct a map, $J_\mathrm{v}$, defined in a neighborhood of $\mathrm{v}$, which continuously deforms $\Omega_{\mathrm{v}}$ to achieve the desired volume while avoiding a significant increase in the weighted multi-perimeter. This result, commonly referred to in the literature as the
``fixing volume variations'' lemma (see, for example,~\cite[Proposition VI.12]{MR420406} and ~\cite[Theorem 29.14]{MR2976521}), is crucial in proving the existence of minimizing clusters.
In our notation, it can be stated as follows.

\begin{figure}[h]
	\centering
		\includegraphics[width=0.9\textwidth]{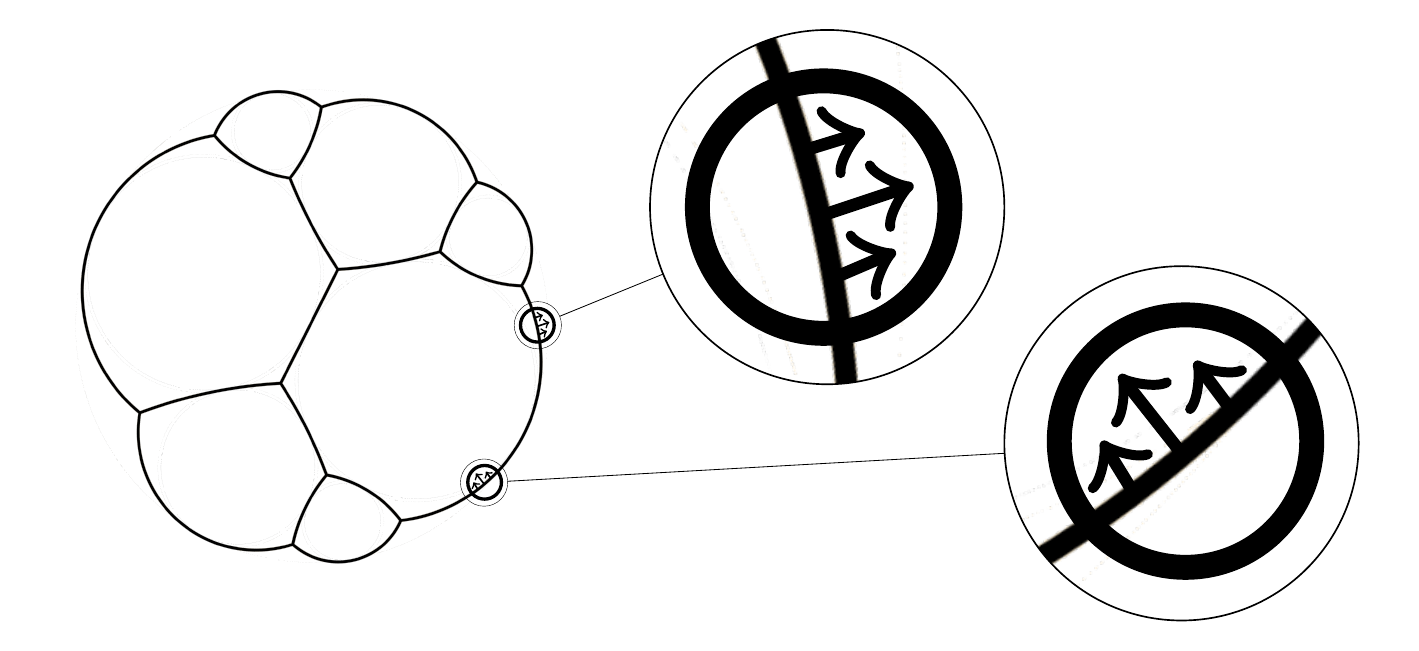}
	\caption{
		A graphic representation of the construction for the proof 
		of Lemma~\ref{lem:fixing-volumes}.
		In each interface, two little balls can be constructed
		to ``inflate'' or ``deflate'' the chamber.
		If the chamber is not linked with the exterior one,
		multiple steps are required to obtain the desired volume.
	}
	\label{fig:fixing-volumes}
\end{figure}

\begin{lemma}
	\label{lem:fixing-volumes}
	There exists $\zeta = \zeta(\mathrm{v}) > 0$ and a continuous map $J_{\mathrm{v}}\colon B(\mathrm{v},\zeta) \subset \mathbb{R}^m_{>0} \to\mathcal{C}^m_\delta(\mathbb{R}^N)$
	such that the following properties hold:
	\begin{itemize}
		\item[{\rm (i)}] $J_{\mathrm{v}}(\mathrm{v}) = \Omega_{\mathrm{v}}$,
			namely $\mathcal{P}_\delta^\omega(\Omega_{\mathrm{v}})
			= \mathcal{I}_\delta^\omega(\mathrm{v})$
		\item[{\rm (ii)}] $J_{\mathrm{v}}(\tilde{\mathrm{v}})\in \mathcal{C}^m_{\delta,\tilde{\mathrm{v}}}(\mathbb{R}^N)$,
			{\it i.e.}, $J_{\mathrm{v}}(\tilde{\mathrm{v}})$
			has volume $\tilde{\mathrm{v}}$
			for any $\tilde{\mathrm{v}} \in B(\mathrm{v},\zeta)$;
		\item[{\rm (iii)}] there exists a constant $C_\mathrm{v} > 0$
			such that 
			\begin{equation}
				\label{eq:loss-perimeter}
				\big|\mathcal{P}^{\omega}_\delta(\Omega_\mathrm{v})-
				\mathcal{P}^{\omega}_\delta(J_{\mathrm{v}}(\tilde{\mathrm{v}}))
				\big|
				\le C_\mathrm{v}\norm{\mathrm{v} - \tilde{\mathrm{v}}},
				\qquad \forall \tilde{\mathrm{v}} \in B(\mathrm{v},\zeta).
			\end{equation}
		\item[{\rm (iv)}] there exists a constant $R_{\mathrm{v}} > 0$ such that 
			\[
				J_{\mathrm{v}}(\tilde{\mathrm{v}})
				\subset B(0,R_{\mathrm{v}}),
				\qquad \forall \tilde{\mathrm{v}} \in B(\mathrm{v},\zeta).
			\]
	\end{itemize}
\end{lemma}

Before presenting the proof, let us make a few comments.
A key property established by Lemma~\ref{lem:fixing-volumes} is expressed in~\eqref{eq:loss-perimeter}, which essentially implies that the loss in the perimeter is linear with respect to the total change in volume.
This is important because the perimeter of the cluster scales as $\norm{\mathrm{v}}^{\frac{N-1}{N}}$ and, consequently,
Lemma~\ref{lem:fixing-volumes} ensures that we can continuously adjust the volume of clusters without a significant increase in the perimeter functional.
Lemma~\ref{lem:fixing-volumes} is very similar to \cite[Theorem 29.14]{MR2976521}, with only minor adjustments needed to accommodate the presence of different weights.
In~\cite[Theorem 29.14]{MR2976521}, using our notations, the weights are given by $\omega_{ij} = 1$ if $i \ne j$ and $\omega_{ij} = 0$ if $i = j$.
However, different weights do not affect the validity of the construction used in the proof, but only change the constant $C_{\mathrm{v}}$ in~\eqref{eq:loss-perimeter}.
Since this result plays a crucial role in our paper,
we outline the key steps of the proof,
highlighting the necessary modifications.

First, let us recall that a 
\emph{one parameter family of diffeomorphisms of $\mathbb{R}^N$} is, for us, a smooth function
\[
	(t,x) \in (-\rho, \rho)\times\mathbb{R}^N \mapsto f(t, x) = f_t(x) \in \mathbb{R}^N, \quad \rho > 0,
\]
such that, for each fixed $|t| < \rho$,
$f_t\colon \mathbb{R}^N \to \mathbb{R}^N$ is a diffeomorphism of $\mathbb{R}^N$,
and $f_0$ is the identity, namely
$f_0(x) = x$ for any $x \in \mathbb{R}^N$.
Now, we can recall the construction used
to adjust the volume of a single chamber,
provided in the following result
(cf.~\cite[Lemma 29.13]{MR2976521}).

\begin{lemma}\label{lemma:maggi2913}
	Let $N\geq2$,
	$\Omega\in\mathcal C^m_\delta(\mathbb R^N)$ be a cluster,
	$z \in\partial^*\Omega^{i}\cap \partial^*\Omega^j$,
	with $0\leq i<j\leq m$, 
	and $\kappa>0$.
	Then, there exist
	$\rho  = \rho (\Omega,z ,\kappa) > 0$,
	$C  = C (N,\rho) > 0$
	and a one-parameter family of diffeomorphisms
	$(f_t)_{|t|<\rho }$ with
	\begin{equation}
		\label{eq:f_t-smallball}
		\overline{\left\{x\in\mathbb R^N\colon x\neq f_t(x)\right\}}\subset B(z ,\rho ),\qquad \forall|t|<\rho ,
	\end{equation}
	which satisfy the following properties:
	\begin{enumerate}
		\item[(i)] for any $|t|<\rho $, we have
			\begin{equation}
				\label{eq:Omegai-inflation}
				\left|\dfrac{\mathrm d}{\mathrm dt}\left|f_t\left({\Omega}^i\right)\cap B(z ,\rho )\right|-1\right|<\kappa,
			\end{equation}
			\begin{equation}
				\label{eq:Omegaj-deflation}
				\left|\dfrac{\mathrm d}{\mathrm dt}\left|f_t\left({\Omega}^{j}\right)\cap B(z ,\rho )\right|+1\right|<\kappa,
			\end{equation}
			\begin{equation}
				\label{eq:Omegak-no-change}
				\left|\dfrac{\mathrm d}{\mathrm dt}\left|f_t\left({\Omega}^k\right)\cap B(z ,\rho )\right|\right|<\kappa,\qquad\mbox{for }k\neq i,j,
			\end{equation}
			\begin{equation}
				\label{eq:Omegak-boundedSecondDeriv}
				\left|\dfrac{\mathrm d^2}{\mathrm dt^2}\left|f_t\left({\Omega}^k\right)\cap B(z ,\rho )\right|\right|<C ,\qquad\mbox{for }0\leq k\leq m;
			\end{equation}
		\item[(ii)] if $\Sigma$ is an $\mathcal H^{N-1}$-rectifiable set in $\mathbb R^N$ and $|t|<\rho $,
			then
			\begin{equation}
				\label{eq:linear-change-perimeter}
				\left|\mathcal H^{N-1}\left(f_t(\Sigma)\right)-\mathcal H^{N-1}\left(\Sigma\right)\right|
				\le C  \, \mathcal{H}^{N-1}(\Sigma)\, |t|.
			\end{equation}
	\end{enumerate}
\end{lemma}

Loosely speaking,
Lemma~\ref{lemma:maggi2913} ensures that,
for any cluster 
$\Omega \in \mathcal{C}^m_\delta(\mathbb{R}^N)$, 
we can exchange small portions of volume between the chambers 
$\Omega^i$ and $\Omega^j$, 
without affecting the other chambers, 
through a local modification in a small ball, 
where this last property is given by~\eqref{eq:f_t-smallball}. 
In fact,~\eqref{eq:Omegai-inflation} can be interpreted as an inflation of the chamber $\Omega^i$, 
since the volume of $f_t(\Omega^i)$ increases 
(recall that $f_0$ is the identity map). 
Similarly,~\eqref{eq:Omegaj-deflation}
indicates a deflation of the chamber $\Omega^j$, 
while~\eqref{eq:Omegak-no-change}
signifies that all other chambers remain substantially unchanged.
Finally,~\eqref{eq:linear-change-perimeter} 
ensures that with these local modifications,
all the $(N-1)$--dimensional measures (and thus the perimeter functional)
change only linearly with respect to $t$.

A graphic representation of two exchanges 
in volumes ensured by Lemma~\ref{lemma:maggi2913}
is shown in Figure~\ref{fig:fixing-volumes},
where a chamber exchanges its volume 
with the exterior one both through inflation
and deflation, separately.
The proof of Lemma~\ref{lem:fixing-volumes}
can be obtained through multiple (but finite) applications of Lemma~\ref{lemma:maggi2913},
as we show in the following.

\begin{proof}[Proof of Lemma~\ref{lem:fixing-volumes}]

	To begin, using the fact that the interfaces of $\Omega_{\mathrm{v}}$ are regular (see Remark \ref{remark:w0}), we can consider a finite set of points
	$(z_j)_{j=1}^k \subset \mathbb{R}^N$
	on the interfaces of the cluster $\Omega_{\mathrm v}$,
	such that each chamber is doubly linked 
	to the exterior one in a finite number of steps.
	By ``doubly linked'', we mean that one can move from the exterior chamber to the selected interior one by passing through a finite number of points on the interfaces,
	and then return to the exterior via different points. 
	An example is illustrated in Figure~\ref{fig:fixing-volumes},
	where a chamber is directly linked to 
	the exterior by two points,
	which are understood as an ``entrance'' and an ``exit.''
	If a chamber $\Omega_\mathrm{v}^i$ is not directly
	linked to the exterior
	({\it i.e.}, $\partial^*\Omega_\mathrm{v}^i \cap \partial^*\Omega_\mathrm{v}^0 = \emptyset$),
	multiple steps through other chambers are required. 
	For more details on this construction,
	particularly a proof that each chamber 
	is connected to the exterior one through a finite number of points,
	we refer to~\cite[Theorem 29.14]{MR2976521}.

	Now, we can exchange a small portion of the volume of each 
	chamber with the exterior, without altering the others.
	Specifically, we apply Lemma~\ref{lemma:maggi2913}
	at each $z_j$, obtaining a small constant $\rho_j > 0$
	and a one-parameter family of diffeomorphisms
	$(f_t^j)_{|t|<\rho_j}$ that satisfy conditions
	(i) and (ii) of Lemma~\ref{lemma:maggi2913},
	where the constants $\kappa$ and $C$ now depend on $j$. 
	At this stage, the ``direction'' of the diffeomorphism
	is predetermined by the setup explained earlier,
	as for each chamber, we need to ensure that
	we can increase or decrease its volume
	by appropriately decreasing or increasing
	the volume of the exterior chamber.

	Let $\rho \coloneqq \min\{\rho_j: j = 1,\dots, k\} > 0$,
	and let us define 
	$\Psi\colon(-\rho,\rho)^k\times\mathbb{R}^N\to\mathbb{R}^N$
	as follows:
	\begin{equation*}
		\Psi(t,x) = \Psi(t^1,\dots, t^k,x) \coloneqq f_{t^1}^1\circ f_{t^2}^2\cdots 
		\circ f_{t^k}^k(x).
	\end{equation*}
	We also define the map $\psi$
	that measures the difference in volume of the chambers 
	after the application of the diffeomorphism $\Psi$.
	Since we take into account also the exterior chamber,
	the map $\psi$ takes values in
	$A_0=\{a\in\mathbb{R}^{m+1}\colon\sum_{i=0}^m a^i=0\}
	\subset\mathbb{R}^{m+1}$.
	Formally, we define
	$\psi\colon(-\rho,\rho)^k\to A_0\subset\mathbb R^{m+1}$ by
	$\psi(t)\coloneqq \big(\psi^0(t),\dots, \psi^m(t)\big)$ where
	\begin{equation*}
		\psi^i(t):=\left|\Psi(t,{\Omega}^i)\cap B(0,R)\right|-\left|\Omega^i \cap B(0,R)\right|=
		\sum_{j = 1}^k
		\Big(
			\left|f_{t^j}^j(\Omega^i) \cap B(z_j,\rho)\right|
			- \left|\Omega^i \cap B(z_j,\rho)\right|
		\Big),
	\end{equation*}
	where $R>0$ is chosen such that $\cup_{j=1}^kB(z_j,\rho)\subset B(0,R)$.
	By construction, $\psi$ is a function of class $C^2$
	with $\psi(0) = 0$
	and by applying~\eqref{eq:Omegai-inflation},~\eqref{eq:Omegaj-deflation}
	and~\eqref{eq:Omegak-no-change}
	on each diffeomorphism $f^j$,
	one can show that 
	$\nabla\psi(0) \in \mathbb{R}^{m+1}\times\mathbb{R}^k$
	is ``near'' to a matrix of rank $m$
	(the maximum rank since $\psi$ takes values in the 
	$m$--dimensional hyperplane $A_0$).
	Moreover, by~\eqref{eq:Omegak-boundedSecondDeriv},
	we have also that $\norm{\nabla^2\psi(0)}$
	is bounded by a constant.
	As a consequence, we can apply the 
	\emph{inverse function theorem}
	(for more details,
	see~\cite[Theorem 29.14, Step two]{MR2976521})
	to obtain a constant $\zeta > 0$
	and a map 
	$\Theta\colon B(0,\zeta) \cap A_0 \subset \mathbb{R}^{m+1}
	\to \mathbb{R}^k$
	such that 
	\begin{equation*}
		\psi(\Theta(a)) = a,
		\qquad \forall a \in  B(0,\zeta) \cap A_0 \subset \mathbb{R}^{m+1}.
	\end{equation*}
	Moreover, there exists a constant $L$ such that
	\begin{equation}
		\label{eq:Theta-control}
		\norm{\Theta(a)} \le L \norm{a},
		\qquad \forall a \in  B(0,\zeta) \cap A_0.
	\end{equation}
	By taking a smaller positive $\zeta$, we can also define a 
	map $\vartheta\colon B(0,\zeta) \subset \mathbb{R}^m \to \mathbb{R}^k$
	as 
	\[
		\vartheta(a) 
		= \Theta\Big(-\sum_{i = 1}^m a^i,a^1,\dots a^m\Big).
	\]
	In other words, the map $\vartheta$ continuously select
	a combination of diffeomorphisms
	such that the variation in volumes of the (interior) chambers
	is the desired one;
	formally, for each $a \in B(0,\zeta) \subset \mathbb{R}^m$,
	we have that the cluster
	\[
		\Big(\Psi(\vartheta(a),\Omega^1),
		\dots,\Psi(\vartheta(a),\Omega^m)\Big)
	\]
	has vectorial volume $\mathrm{v} + a$.
	Hence, by a simple translation, we define 
	$J_\mathrm{v}\colon B(\mathrm{v},\zeta) \to \mathcal{C}_\delta^m(\mathbb{R}^{N})$
	as follows:
	\[
		J_\mathrm{v}(\tilde{\mathrm{v}})
		\coloneqq 
		\Big(\Psi(\vartheta(\tilde{\mathrm{v}}- \mathrm{v}),\Omega^1),
		\dots,\Psi(\vartheta(\tilde{\mathrm{v}}- \mathrm{v}),\Omega^m)\Big).
	\]
	We notice that $J_\mathrm{v}$ is a continuous map,
	since it is defined by the combination
	of a continuous map with a combination 
	of diffeomorphisms,
	and that (i) and (ii) are held by construction. 
	Since $\Omega_{\mathrm{v}}$
	is bounded, and each diffeomorphism 
	moves only the points inside 
	the small balls $B(z_j,\rho)$,
	also (iv) holds for some $R_{\mathrm{v}} > 0$.

	To show (iii), we need to apply~\eqref{eq:linear-change-perimeter}
	on each diffeomorphism that defines $\Psi$.
	Let $C_j > 0$ be the constant given by Lemma~\ref{lemma:maggi2913}
	for each $f^j$, and set $C = \max\{C_j: j = 1,\dots,k\}$.
	Any $(N-1)$--dimensional rectifiable set $\Sigma \subset \mathbb{R}^N$
	is modified by $\Psi$ through the $k$ diffeomorphisms
	$(f^j)_{j = 1}^k$,
	so from~\eqref{eq:linear-change-perimeter} we obtain
	\begin{align*}
		\left|
		\mathcal{H}^{N-1}\big(\Psi(t,\Sigma)\big)
		- \mathcal{H}^{N-1}\big(\Sigma\big)
		\right|
&\le 
\sum_{j = 1}^k
\left| 
\mathcal{H}^{N-1}\big(f^j_{t^j}(\Sigma)\cap B(z_j,\rho)\big)
- \mathcal{H}^{N-1}\big(\Sigma \cap B(z_j,\rho) \big)
\right| \\
&\le C \sum_{j = 1}^k |t^j| \mathcal{H}^{N-1} \big(\Sigma \cap B(z_j,\rho) \big) 
\le C \, k\, \mathcal{H}^{N-1}(\Sigma) \norm{t}.
	\end{align*}
	Therefore, by using also~\eqref{eq:Theta-control},
	for every interface $\Sigma_{ij} = \partial^*\Omega_{\mathrm{v}}^i \cap \partial^*\Omega_{\mathrm{v}}^j$
	we have 
	\begin{align*}
		\left|
		\mathcal{H}^{N-1}\big(\Psi(\vartheta(\tilde{\mathrm{v}}-\mathrm{v}),\Sigma_{ij})\big)
		- \mathcal{H}^{N-1}\big(\Sigma_{ij}\big)
		\right|
&\le C k \mathcal{H}^{N-1}(\Sigma_{ij})
\norm{\vartheta(\tilde{\mathrm{v}}-\mathrm{v})}
\\
&\le C L k \mathcal{H}^{N-1}(\Sigma_{ij})
\norm{\tilde{\mathrm{v}}-\mathrm{v}},
\qquad \forall \tilde{\mathrm{v}} \in B(\mathrm{v},\zeta). 
	\end{align*}
	As a consequence, by setting 
	\begin{equation*}
		C_{\mathrm{v}}=CLk\sum_{i,j=0}^m\left(\omega_{ij}\mathcal H^{N-1}(\Sigma_{ij})\right),
	\end{equation*}
	and recalling the definition of the weighted multiperimeter
	in~\eqref{eq:def-Pomega},
	we finally obtain
	\begin{multline*}
		\left|
		\mathcal{P}^\omega_\delta(\Omega_{\mathrm{v}})-\mathcal \mathcal{P}^\omega_\delta(J_{\mathrm{v}}(\tilde{\mathrm{v}}))
		\right|
		\le
		\sum_{i,j=0}^{m}\omega_{ij}
		\left|
		\mathcal{H}^{N-1}(\Psi(\vartheta(\tilde{\mathrm{v}}-\mathrm{v}),\Sigma_{ij}))-
		\mathcal{H}^{N-1}(\Sigma_{ij})
		\right|
		\\
		\le C L k
		\sum_{i,j=0}^{m}
		\Big(\omega_{ij}
			\mathcal{H}^{N-1}(\Sigma_{ij})
		\Big)
		\norm{\tilde{\mathrm{v}}-\mathrm{v}}
		= C_{\mathrm{v}} \norm{\tilde{\mathrm{v}}-\mathrm{v}},
		\qquad \forall \tilde{\mathrm{v}} \in B(\mathrm{v},\zeta),
	\end{multline*}
	and this concludes the proof of item (iii).
\end{proof}

\begin{remark}
	\label{remark:eta-inversetheorem}
	From the proof of Lemma~\ref{lem:fixing-volumes},
	we can understand why it is not possible to provide
	an explicit estimate for the parameter $\zeta > 0$,
	whose existence is ensured by the inverse function theorem.
	Moreover, even if one succeeds in applying a quantitative
	version of the inverse function theorem, 
	such an estimate would depend also on the radii $\rho_j$
	where individual inflations and deflations are applied.
	These radii, in turn, rely on selecting
	the most appropriate points $(z_j)_{j=1}^k$ in the construction,
	which seems impossible without a characterization 
	of the minimizer $\Omega_{\mathrm{v}}$.
\end{remark}

Thanks to Lemma~\ref{lem:fixing-volumes}, we can select a family of ``almost''
weighted isoperimetric clusters on each tangent space 
of our manifold.
However, this is still not enough 
for constructing the photography map
by simply applying a projection
on the manifold of these clusters 
through the exponential map.
There are two main technical issues in this construction.
The first is that the diameters of the clusters provided by $J_{\mathrm{v}}$
may exceed the injectivity radius of the manifold,
leading to complications during the
projection through the exponential map.
The second challenge arises from the distortion caused by the curvature
of the manifold.
Indeed, it could happen that 
for any $\tilde{\mathrm{v}} \in B(\mathrm{v},\zeta)$
the vectorial volume of the projection
of $J_{\mathrm{v}}(\tilde{\mathrm{v}})$
is not $\mathrm{v}$, as required.
This is because the ball $B(\mathrm{v},\zeta)$
could not be sufficiently large to fully compensate the distortion 
caused by the curvature, and we are not able to 
provide a lower bound on the constant $\zeta$,
which is given by an application of the inverse function theorem.
As previously explained in Section~\ref{subsec:photography},
we heavily exploit the rescaling invariance of the 
isoperimetric problem to overcome this issue.

For any $\alpha > 0$,
and $\Omega \in \mathcal{C}^{m}_{\delta,\mathrm{v}}(\mathbb{R}^N)$,
let us denote by
$\alpha\Omega\in\mathcal{C}^\omega_{\delta,\alpha\mathrm{v}}$
the $m$-cluster of volume $\alpha \mathrm{v}$
obtained by a proper rescaling of $\Omega$.
Since we are on the $N$-dimensional Euclidean space,
this means that
\[
	x \in \Omega^i 
	\text{ if and only if }\alpha^{1/N} x \in (\alpha\Omega)^i.
\]
Using this notation,
the weighted multi-perimeter functional 
satisfies the following scaling law:
\begin{equation}
	\label{eq:g-law-perimeter}
	\mathcal{P}^{\omega}_\delta(\alpha\Omega)
	= \alpha^{\frac{N-1}{N}}
	\mathcal{P}^{\omega}_{\delta}(\Omega),
	\qquad \forall \alpha > 0,
	\, \forall \Omega \in \mathcal{C}^{m}_{\delta}(\mathbb{R}^N).
\end{equation}
Thanks to this scaling law,
we can ensure that
the clusters obtained by composing the map
$J_{\mathrm{v}}$, given by Lemma~\ref{lem:fixing-volumes},
with a rescaling of a small factor $\alpha$
are inside a ball whose 
radius is smaller than the injectivity radius of the manifold $M$.
Formally,
we proceed as follows.
By the translation invariance property of the
$\omega$-weighted multi-perimeter in $\mathbb{R}^N$,
we can assume that the origin is the barycenter of $\Omega_{\mathrm{v}}$,
meaning that
\[
	\int_{\mathbb{R}^N} x\cdot\chi_{\cup_{i=1}^m \Omega_{\mathrm{v}}^i} \mathrm{d} x = 0.
\]
Recalling the definition of the constant $R_{\mathrm{v}}$
in Lemma~\ref{lem:fixing-volumes},
there exists $\alpha_0 = \alpha_0(M,g,\mathrm{v}) > 0$
such that 
\begin{equation}
	\label{eq:def-alpha-zero}
	\alpha^{1/N}R_{\mathrm{v}} < \inj_M,
	\qquad \forall \alpha \in (0,\alpha_0).
\end{equation}
In other words, as long as $\alpha <  \alpha_0$,
we are sure that 
\[
	\alpha J_{\mathrm{v}}(\tilde{\mathrm{v}}) \subset B(0,\inj_M),
	\qquad \forall \tilde{\mathrm{v}} \in B(\mathrm{v},\zeta),
\]
so that their projection into the manifold
at any point $p \in M$ are well defined and
they ``look like'' $\alpha J_{\mathrm{v}}(\tilde{\mathrm{v}})$.

To proceed further, let us formally define the projection map.
Let $\mathcal{C}^m_{\delta}(B(0,\inj_M))
\subset \mathcal{C}^m_{\delta}(\mathbb{R}^{N})$
the subset of $m$-clusters in $\mathbb{R}^N$
whose every interior chamber is contained in 
the Euclidean ball centered in the origin and with radius $\inj_M$.
Then, we define the projection map
$\Pi\colon M \times \mathcal{C}^m_\delta(B(0,\inj_M)))\to\mathcal{C}^m_g(M)$
as follows:
\begin{equation}
	\label{eq:def-projection-map}
	\Pi(p,\Omega) \coloneqq 
	\Big(
		\bigcup_{x \in \Omega^1} \exp_p(x),
		\bigcup_{x \in \Omega^2} \exp_p(x),
		\ldots,
		\bigcup_{x \in \Omega^m} \exp_p(x)
	\Big).
\end{equation}
To simplify our notation, we will write every chamber as
\[
	\exp_p(\Omega^i) = \bigcup_{x \in \Omega^i}\exp_p(x)
	\subset M.
\]
We remark that the projection map is well defined
since we are considering just
the Euclidean clusters contained in $B(0,\inj_M)$.
Indeed, we have
\[
	\Omega^i \cap \Omega^j = \emptyset
	\implies \exp_p(\Omega^i) \cap \exp_p(\Omega^j) = \emptyset.
\]

Before the next result, we need to establish further notations.
For any $\alpha > 0$
we define $J_{\mathrm{v},\alpha}$
as the composition of $J_{\mathrm{v}}$
with a rescaling of a factor $\alpha$
on the volumes of the chambers;
formally, we get
\[
	J_{\mathrm{v},\alpha}\colon B(\mathrm{v},\zeta) \to 
	\mathcal{C}^{m}_{\delta}(\mathbb{R}^N),
\]
and 
\[
	J_{\mathrm{v},\alpha}(\tilde{\mathrm{v}})
	\coloneqq
	\alpha J_{\mathrm{v}}(\tilde{\mathrm{v}}).
\]
As a consequence, 
since $J_{\mathrm{v}}(\tilde{\mathrm{v}})$ has volume 
$\tilde{\mathrm{v}}$ by definition of $J_{\mathrm{v}}$,
we have 
\[
	\mathcal{V}_{\delta}(J_{\mathrm{v},\alpha}(\tilde{\mathrm{v}})) 
	= \alpha\, \tilde{\mathrm{v}},
	\qquad \forall \tilde{\mathrm{v}} \in B(\mathrm{v},\zeta),
\]
and 
\[
	\mathcal{P}^{\omega}_{\delta}
	(J_{\mathrm{v},\alpha}\big(\tilde{\mathrm{v}})\big)
	= \alpha^{\frac{N-1}{N}}
	\mathcal{P}^{\omega}_{\delta}
	(J_{\mathrm{v}}\big(\tilde{\mathrm{v}})\big),
	\qquad \forall \tilde{\mathrm{v}} \in B(\mathrm{v},\zeta).
\]
Now, we are ready to state the main technical lemma of this section.
\begin{lemma}
	\label{lem:g}
	There exists
	$\alpha_1 = \alpha_1(M,g,\mathrm{v}) \in (0,\alpha_0)$
	such that, 
	for every $\alpha \in (0,\alpha_1)$ 
	there exists a continuous map
	$\Phi_{\mathrm{v},\alpha}\colon M \to B(\mathrm{v},\zeta)$
	such that
	\begin{equation}
		\label{eq:def-Phi}
		\mathcal{V}_g\left(
			\Pi\left(p,J_{\mathrm{v},\alpha}(\Phi_{\mathrm{v},\alpha}(p))\right)
		\right)
		= \alpha\,\mathrm{v},
		\qquad \forall p \in M.
	\end{equation}
\end{lemma}
\begin{proof}
	The proof is based on the fact that, as $\alpha \to  0$, 
	the clusters are given by
	$J_{\mathrm{v},\alpha}$ and are inside the ball
	$B(0,\alpha^{\frac{1}{N}} R_{\mathrm{v}})$.
	Consequently, their distortion through the exponential map 
	becomes sufficiently small.
	This last property is a consequence of the fact 
	that the exponential map
	is bi-Lipschitz, which Lipschitz constants
	converging to $1$ as the radius goes to zero
	(see, for example,~\cite[Corollary 2.4, p. 145]{MR1480173}).
	The details of this construction are presented below.

	Since $M$ is a compact manifold, there exists
	a constant $K > 0$ such that 
	\[
		- K \le \sec_{g}(p)(w,z) \le K,
		\qquad \forall p \in M,\, \forall w,z \in T_pM,
	\]
	where $\sec_g(p)$ denotes the sectional curvature at $p \in M$.
	Using standard estimates for the exponential map (see, e.g.,~\cite[Corollary 2.4, p. 145]{MR1480173}), we have, for any $p \in M$ and for sufficiently small $R > 0$, that:
	\begin{equation}
		\label{eq:exp-Lipschitz}
		\sup_{r \in B(0,R)}
		|\mathrm{D}(\exp_p(r))|
		\le \frac{\sinh(\sqrt{K}R)}{\sqrt{K} R},
		\quad\text{and}\quad
		\sup_{r \in B(0,R)}
		|\mathrm{D}(\exp_p(r))^{-1}|
		\le \frac{\sin(\sqrt{K}R)}{\sqrt{K} R}.
	\end{equation}
	This implies that if $\Omega^i \subset B(0,R) \subset \mathbb{R}^N$ is an open set with Euclidean volume $\mathrm{v}^i \in \mathbb{R}_{>0}$, we conclude that the volume of its projection fulfills the following estimates:
	\[
		\left(\frac{\sin(\sqrt{K}R)}{\sqrt{K}R}\right)^{N}\mathrm{v}^i
		\le
		\mathcal{V}_g(\exp_p(\Omega^i))
		\le
		\left(\frac{\sinh(\sqrt{K}R)}{\sqrt{K}R}\right)^{N}\mathrm{v}^i,
		\qquad \forall p \in M.
	\]
	By Lemma~\ref{lem:fixing-volumes} and by definition of  $J_{\mathrm{v},\alpha}$, we know 
	\[
		J_{\mathrm{v},\alpha}(\tilde{\mathrm{v}})
		\subset B(0,\alpha^{1/N} R_{\mathrm{v}}),
		\qquad \tilde{\mathrm{v}} \subset B(\mathrm{v},\zeta),
	\]
	and, assuming that $\alpha \in (0,\alpha_0)$,
	by~\eqref{eq:def-alpha-zero} we obtain
	$J_{\mathrm{v},\alpha}(\tilde{\mathrm{v}})\subset B(0,\inj_{M})$.
	Therefore, we know that for any
	$\tilde{\mathrm{v}} \subset B(\mathrm{v},\zeta)$
	and for any $i = 1,\dots,m$, we get
	\[
		\left(\frac{\sin(\sqrt{K}\alpha^{1/N}R_{\mathrm{v}})}{\sqrt{K}\alpha^{1/N} R_{\mathrm{v}}}\right)^{N}
		\alpha\,\tilde{\mathrm{v}}^i
		\le
		\mathcal{V}_g\Big(\big(
		\Pi(p,J_{\mathrm{v},\alpha}(\tilde{\mathrm{v}}))\big)^i\Big)
		\le
		\left(\frac{\sinh(\sqrt{K}\alpha^{1/N}R_{\mathrm{v}})}{\sqrt{K}\alpha^{1/N} R_{\mathrm{v}}}\right)^{N}
		\alpha\,\tilde{\mathrm{v}}^i,
		\qquad \forall p \in M.
	\]
	Now, let us notice that the $m$--dimensional cube below
	\[
		\mathrm{v} + \left(-\frac{\zeta}{\sqrt{N}},+\frac{\zeta}{\sqrt{N}}\right)^m
		\coloneqq
		\left(\mathrm{v}^1 - \dfrac{\zeta}{\sqrt{N}}, \mathrm{v}^1 + \dfrac{\zeta}{\sqrt{N}}\right)
		\times \dots \times
		\left(\mathrm{v}^m - \dfrac{\zeta}{\sqrt{N}}, \mathrm{v}^m + \dfrac{\zeta}{\sqrt{N}}\right)
	\]
	is a subset of the ball $B(\mathrm{v},\zeta)$. This implies that, for every component of $\mathrm{v}$, we can change it by a quantity between $\pm \zeta/\sqrt{N}$ with the possibility of changing within the same range as all the other components.
	Therefore, it suffices to prove the existence of  $\alpha_1(M,g,\mathrm{v}) > 0$ such that for every $\alpha \in (0,\alpha_1)$ and for every $i = 1,\dots,m$,  we find
	\begin{equation}
		\label{eq:g-estimate-1}
		\left(\frac{\sin(\sqrt{K}\alpha^{1/N}R_{\mathrm{v}})}{\sqrt{K}\alpha^{1/N} R_{\mathrm{v}}}\right)^{N}
		\alpha\left(\mathrm{v}^i + \frac{\zeta}{\sqrt{N}}\right)
		> \alpha \mathrm{v}^i,
	\end{equation}
	and
	\begin{equation}
		\label{eq:g-estimate-2}
		\left(\frac{\sinh(\sqrt{K}\alpha^{1/N}R_{\mathrm{v}})}{\sqrt{K}\alpha^{1/N} R_{\mathrm{v}}}\right)^{N}
		\alpha\left(\mathrm{v}^i - \frac{\zeta}{\sqrt{N}}\right)
		< \alpha \mathrm{v}^i.
	\end{equation}
	If these last two inequalities hold, then 
	exploiting also the continuity of the map $J_{\mathrm{v},\alpha}$
	we have that for every $p \in M$, it should exists 
	a volume $\tilde{\mathrm{v}}_{p,\alpha}$ in the cube 
	$\mathrm{v} + (-\zeta/\sqrt{N},+\zeta/\sqrt{N})^m \subset B(\mathrm{v},\zeta)$
	such that 
	\[
		\mathcal{V}_g(
		\Pi(p,J_{\mathrm{v},\alpha}(\tilde{\mathrm{v}}_{p,\alpha}))
		)
		= \alpha\mathrm{v}.
	\]
	Moreover, 
	by standard results on Riemannian geometry 
	we know that the exponential map is continuous
	with respect to the application points.
	In our case, this particularly means that
	for every $\Omega \in C^m_\delta(B(0,\inj_M))$ and $(p_k)_{k\in\mathbb N} \subset M$ a sequence of points satisfying $\lim_{k \to \infty}\dist_{g}(p_k,p) = 0$, we have 
	\[
		\lim_{k \to \infty}
		\dist\left(\Pi(p_k,\Omega), \Pi(p,\Omega)\right) = 0.
	\]
	As a consequence, the volumes
	$\tilde{\mathrm{v}}_{p,\alpha}$ can be chosen continuously 
	with respect to $p$, so that defining
	$\Phi_{\mathrm{v},\alpha}\colon M \to B(\mathrm{v},\zeta)$
	as $\Phi_{\mathrm{v},\alpha}(p) = \tilde{\mathrm{v}}_{p,\alpha}$,
	it is a continuous function.

	Therefore, to end the proof 
	it suffices to show 
	the existence of $\alpha_1 > 0$
	such that for every $\alpha \in (0,\alpha_1)$
	the inequalities~\eqref{eq:g-estimate-1}
	and~\eqref{eq:g-estimate-2} hold.
	The key idea is that we are allowed to increase the volume 
	of the Euclidean cluster by a factor that is linear
	with respect to $\alpha$, while the loss
	due to the distortion of the exponential map is 
	of the order of $\alpha^{1+2/N}$.
	Since this is a crucial step of the whole paper,
	let us show this computation in detail.
	By the Taylor expansion of the sine function,
	we get
	\[
		\left(\frac{\sin x}{x}\right)^N
		= 1 - N \frac{x^2}{6} + \mathrm{o}(x^3),
	\]
	hence
	\[
		\left(\frac{\sin(\sqrt{K}\alpha^{1/N}R_{\mathrm{v}})}{\sqrt{K}\alpha^{1/N} R_{\mathrm{v}}}\right)^{N}
		= 1 - \frac{N K R_\mathrm{v}^2}{6} \alpha^{2/N} 
		+ \mathrm{o}(\alpha^{2/N}).
	\]
	Let us set $c = (NKR_{\mathrm{v}}^2)/6$.
	In order to prove~\eqref{eq:g-estimate-1}
	we need to show that, for every $\alpha$ sufficiently small,
	we have
	\[
		\alpha \left(\mathrm{v}^i + \frac{\zeta}{\sqrt{N}}\right)
		\big(1 - c\, \alpha^{2/N} + \mathrm{o}(\alpha^{2/N})\big)
		> \alpha \mathrm{v}^i,
	\]
	or, equivalently,
	\[
		\alpha \frac{\zeta}{\sqrt{N}}
		- c\left(\mathrm{v}^i + \dfrac{\zeta}{\sqrt{N}}\right) \alpha^{1+2/N}
		+ \mathrm{o}(\alpha^{1 + 2/N}) > 0.
	\]
	As previously stated, since $\alpha^{1 + 2/N} = \mathrm{o}(\alpha)$,
	there exists $\alpha_1^i(M,g,\mathrm{v}^i) > 0$
	such that for every $\alpha \in (0,\alpha_1^i)$
	the inequality~\eqref{eq:g-estimate-1} holds.
	Taking the minimum among the components of the clusters,
	namely over $i = 1,\dots,m$, we have the desired result.
	For~\eqref{eq:g-estimate-2}, a similar computation leads to
	the inequality
	\[
		- \alpha \frac{\zeta}{\sqrt{N}}
		+ c\left(\mathrm{v}^i - \frac{\zeta}{\sqrt{N}}\right) \alpha^{1+2/N}
		+ \mathrm{o}(\alpha^{1 + 2/N}) < 0,
	\]
	which similarly holds if $\alpha$ is sufficiently small, and we are done.
\end{proof}

Thanks to Lemma~\ref{lem:g}, we are able to construct
a map that associates to each point of the manifold
a cluster with volume $\alpha \mathrm{v}$,
if $\alpha$ is sufficiently small.
In order to end the construction for the proof of
Proposition~\ref{prop:photography},
we need to estimate the weighted multi-perimeter
of these clusters, and then obtain an analogous estimate
for the vectorial functions that approximate the cluster in the sense
of the $\Gamma$--convergence.
Let us introduce the following notation.
For any $\alpha \in (0,\alpha_1)$,
let $G_{\mathrm{v},\alpha}\colon M \to \mathcal{C}^m_{g,\alpha\mathrm{v}}(M)$
be the projection of the composition of $\Phi_{\mathrm{v},\alpha}$
with $J_{\mathrm{v},\alpha}$, hence
\begin{equation}
	\label{eq:def-G}
	G_{\mathrm{v},\alpha}(p)
	\coloneqq
	\Pi\Big(
		p,J_{\mathrm{v},\alpha}\big(\Phi_{\mathrm{v},\alpha}(p)\big)
	\Big).
\end{equation}

\begin{remark}
	\label{rem:boundenssGvalpha}
	By construction, 
	for any $\alpha \in (0,\alpha_1)$,
	we have 
	\[
		J_{\mathrm{v},\alpha}\big(\Phi_{\mathrm{v},\alpha}(p)\big)
		\subset B(0,\alpha^{\frac{1}{N}} R_\mathrm{v}).
	\]
	As a consequence,
	by exploiting again the compactness of $M$,
	this implies that there exists a constant $R_{g,\mathrm{v}}$
	such that 
	\begin{equation}
		\label{eq:boundenssGvalpha}
		G_{\mathrm{v},\alpha}(p) \subset
		B(p,\alpha^{\frac{1}{N}} R_{g,\mathrm{v}}),
		\qquad \forall p \in M.
	\end{equation}
\end{remark}

Thanks to the definition of $J_{\mathrm{v}}$
and the scaling law of the perimeter functional,
we have the following result.
\begin{lemma}
	\label{lem:estimate-photo-noEpsilon}
	There exists $\alpha^*_1 = \alpha^*_1(M,g,\mathrm{v}) \in (0,\alpha^*_1)$
	and a function
	$\tau_{\mathrm{v}}\colon (0,\alpha^*_1) \to \mathbb{R}$
	such that
	\begin{equation}
		\label{eq:tau-zero-infinitesimal-alpha}
		\lim_{\alpha \to 0}
		\frac{\tau_{\mathrm{v}}(\alpha)}{\mathcal{I}^{\omega}_\delta(\alpha\mathrm{v})} = 0,
	\end{equation}
	and, for every $\alpha \in (0,\alpha^*_1)$ we have
	\begin{equation}
		\label{eq:estimate-photo-noEpsilon}
		\mathcal{P}^{\omega}_{g}(G_{\mathrm{v,\alpha}}(p))
		< \mathcal{I}^{\omega}_{\delta}(\alpha \mathrm{v}) + \tau_{\mathrm{v}}(\alpha),
		\qquad \forall p \in M.
	\end{equation}
\end{lemma}
\begin{proof}
	By definition of $J_{\mathrm{v},\alpha}$,
	and the rescaling law of the perimeter in the Euclidean space,
	we know that
	\[
		\mathcal{P}^{\omega}_{\delta}(J_{\mathrm{v},\alpha}(\mathrm{v}))
		= \mathcal{P}^{\omega}_{\delta}(\alpha \Omega_{\mathrm{v}})
		= \mathcal{I}^{\omega}_{\delta}(\alpha \mathrm{v}).
	\]
	Moreover, by~\eqref{eq:loss-perimeter}, we have
	\begin{equation}
		\label{eq:estimate-photo-noEpsilon-proof1}
		\mathcal{P}^{\omega}_{\delta}(J_{\mathrm{v},\alpha}(\tilde{\mathrm{v}}))
		\le
		\mathcal{P}^{\omega}_{\delta}(\alpha \Omega_{\mathrm{v}})
		+ \alpha C_{\mathrm{v}} \norm{\tilde{\mathrm{v}}-\mathrm{v}}
		\le
		\mathcal{I}^{\omega}_{\delta}(\alpha \mathrm{v})
		+ \alpha C_{\mathrm{v}} \zeta,
		\qquad \forall \tilde{\mathrm{v}} \in B(\mathrm{v},\zeta).
	\end{equation}
	Now, we again use the bi-Lipschitzianity of the exponential map, namely, the estimates given by \eqref{eq:exp-Lipschitz}. Since the multi-perimeter is defined through the $(N-1)$-dimensional
	Hausdorff measure,
	by combining~\eqref{eq:estimate-photo-noEpsilon-proof1}
	with~\eqref{eq:exp-Lipschitz},
	and recalling the definition of $G_{\mathrm{v},\alpha}$
	given in~\eqref{eq:def-G},
	we infer that, if $\alpha \in (0,\alpha_1)$,
	the following inequality holds:
	\[
		\mathcal{P}^{\omega}_g\big(G_{\mathrm{v},\alpha}(p)\big)
		\le
		\left(\frac{\sinh(\sqrt{K}\alpha^{1/N}R_{\mathrm{v}})}{\sqrt{K}\alpha^{1/N} R_{\mathrm{v}}}\right)^{N-1}
		\big(
			\mathcal{I}^{\omega}_{\delta}(\alpha\mathrm{v})
			+ \alpha C_\mathrm{v}\zeta
		\big),
		\qquad \forall p \in M.
	\]
	By using the Taylor expansion of the hyperbolic sine function, 
	setting $c = (N-1)KR_{\mathrm{v}}^2/6$,
	and recalling that
	$\mathcal{I}^{\omega}_{\delta}(\alpha\mathrm{v})
	= \alpha^{\frac{N-1}{N}}\mathcal{I}^{\omega}_{\delta}(\mathrm{v})$,
	we obtain that
	\begin{align*}
		\mathcal{P}^{\omega}_g\big(G_{\mathrm{v},\alpha}(p)\big)
		&\leq 
		\big(1 + c \alpha^{2/N} + \mathrm{o}(\alpha^{2/N})\big)
		\big(
			\mathcal{I}^{\omega}_{\delta}(\alpha\mathrm{v})
			+ \alpha C_\mathrm{v}\zeta
		\big)\\
		&=
		\alpha^{\frac{N-1}{N}}
		\mathcal{I}^{\omega}_{\delta}(\mathrm{v})
		+ \alpha C_\mathrm{v}\zeta
		+ \alpha^{\frac{N+1}{N}} c \mathcal{I}^{\omega}_{\delta}(\mathrm{v})
		+\alpha^{\frac{N+2}{N}} c C_{\mathrm{v}}\zeta
		+ \mathrm{o}(\alpha^{\frac{N+2}{N}}).
	\end{align*}
	Therefore, setting
	\[
		\tau_{\mathrm{v}}(\alpha) \coloneqq
		\alpha (\mathrm{C}_{\mathrm{v}} \zeta + 1),
	\]
	there exists $\alpha^*_1 = \alpha^*_1(M,g,\mathrm{v}) \in (0,\alpha_1)$
	such that
	\[
		\tau_{\mathrm{v}}(\alpha) >
		\alpha C_\mathrm{v}\zeta
		+ \alpha^{\frac{N+1}{N}} c\, \mathcal{I}^{\omega}_{\delta}(\mathrm{v})
		+\alpha^{\frac{N+2}{N}} c\, C_{\mathrm{v}}\zeta
		+ \mathrm{o}(\alpha^{\frac{N+2}{N}}),
		\qquad \forall \alpha \in (0,\alpha^*_1).
	\]
	Therefore,~\eqref{eq:tau-zero-infinitesimal-alpha}
	and~\eqref{eq:estimate-photo-noEpsilon} hold, which concludes the proof.
\end{proof}

Finally, we are ready to prove Proposition~\ref{prop:photography},
thus performing the~\ref{item:s1} for the proof of the main theorem.

\begin{proof}[Proof of Proposition~\ref{prop:photography}]
	The proof employs all the above construction of this section
	and the $\Gamma$--convergence
	of the Allen-Cahn functional $\mathcal{E}_{\varepsilon,\alpha\mathrm{v}}$
	to the weighted multi-perimeter 
	(see Proposition~\ref{prop:gamma-convergence}).

	Let us choose $\alpha^*_1 > 0$
	and the function $\tau_{\mathrm{v}}\colon (0,\alpha^*_1) \to \mathbb{R}$
	as in Lemma~\ref{lem:estimate-photo-noEpsilon}.
	For any $\varepsilon > 0$
	and $\alpha \in (0,\alpha_1)$, let $R_{\varepsilon}\colon \mathcal{C}^m_{g,\alpha\mathrm{v}}(M)
	\to W^{1,2}_{\alpha\mathrm{v}}(M)$ be the map
	that select the element corresponding to $\varepsilon$
	of the recovery sequence of a cluster,
	whose existence is ensured by the $\limsup$ property
	of Proposition~\ref{prop:gamma-convergence}.
	With this notation, we define the photography map
	$ \varphi_{\varepsilon,\alpha\mathrm{v}}\colon M \to W^{1,2}_{\alpha\mathrm{v}}(M) $
	as follows:
	\begin{equation*}
		\varphi_{\varepsilon,\alpha\mathrm{v}}(p)
		\coloneqq R_\varepsilon\big(G_{\mathrm{v},\alpha}(p)\big),
	\end{equation*}
	where $G_{\mathrm{v},\alpha}\colon M \to \mathcal{C}^m_{g,\alpha\mathrm{v}}(M)$
	is defined in~\eqref{eq:def-G}.
	By Proposition~\ref{prop:gamma-convergence}
	and Lemma~\ref{lem:estimate-photo-noEpsilon}
	we have
	\[
		\lim_{\varepsilon \to 0}
		\mathcal{E}_{\varepsilon,\alpha\mathrm{v}}\big(
		\varphi_{\varepsilon,\alpha\mathrm{v}}(p)\big)
		= \mathcal{P}^{\omega}_g(G_{\mathrm{v},\alpha}(p))
		< \mathcal{I}_{\delta}^\omega(\alpha\mathrm{v}) + \tau_{\mathrm{v}}(\alpha),
		\qquad \forall p \in M.
	\]
	As a consequence, 
	using again the compactness hypothesis on the manifold $M$,
	there exists $\varepsilon^*_1 = \varepsilon^*_1(\alpha,\mathrm{v}) > 0$
	such that, for any $\varepsilon \in (0,\varepsilon^*_1)$
	and $p \in M$ we have
	\[
		\mathcal{E}_{\varepsilon,\alpha\mathrm{v}}\big(
		\varphi_{\varepsilon,\alpha\mathrm{v}}(p)\big)
		< \mathcal{I}^\omega_\delta(\alpha\mathrm{v}) + \tau_{\mathrm{v}}(\alpha).
	\]
	and the proof is finished.
\end{proof}

\section{Barycenter Map}
\label{sec:barycenter}
In this section, we construct a barycenter map,
defined on the sublevel that contains the image of the photography
map, provided $\alpha$ and $\varepsilon$ sufficiently small.
In other words, we provide a proof of 
Proposition~\ref{prop:barycenter},
thus performing~\ref{item:s2}.

As a first step, we prove Proposition~\ref{prop:almost-min-concentration}, thereby extending the concentration result from Proposition~\ref{prop:concentration-almost-minimizers} to the
``almost minimizers'' of the functional $\mathcal{E}_{\varepsilon,\alpha\mathrm{v}}$.
This is achieved by exploiting the equicoerciveness and $\Gamma$--convergence of the volume-rescaled energy functional.

\begin{proof}[Proof of Proposition~\ref{prop:almost-min-concentration}]
	As usual, let $\mathrm{v} \in \mathbb{R}^m_{> 0}$ be fixed,
	and let $\mu > 0$ be given by
	Proposition~\ref{prop:concentration-almost-minimizers}.
	By contradiction, we assume that there exist
	$\theta_0\in (0,1)$ 
	and an infinitesimal sequence $(\alpha_k)_{k\in\mathbb N} \subset \mathbb{R}_{>0}$
	such that,
	for any $k \in \mathbb{N}$,
	there exists two sequences $(\varepsilon_{k_\ell})_{\ell \in \mathbb{N}}
	\subset \mathbb{R}_{> 0}$
	and 
	$(u_{k_\ell})_{\ell\in\mathbb N}
	\subset 
	W_{\alpha_k \mathrm{v}}^{1,2}(M,\mathbb{R}^m)$
	such that
	\begin{equation}
		\label{eq:boh-0}
		\mathcal{E}_{\varepsilon_{k_\ell},\alpha_k\mathrm{v}}(u_{k_\ell})
		\le \sigma_{\mathrm{v}}(\alpha_k),
		\qquad \forall \ell \in \mathbb{N},
	\end{equation}
	and
	\begin{equation}
		\label{eq:boh-1}
		\int_{M \setminus B_g(p,\mu_k)}
		\norm{u_{k_\ell}} \ud v_{g}
		> \theta_0 \norm{\mathrm{v}} \alpha_k,
		\qquad
		\forall  p \in M \; \text{and} \; \forall\ell\in\mathbb{N},
	\end{equation}
	where we set $\mu_k\coloneqq \mu(\alpha_k\norm{\mathrm{v}})^{1/N}$ 
	for simplicity.
	By~\eqref{eq:boh-0}
	and the equicoerciveness of the functionals
	$(\mathcal{E}_{\varepsilon_{k_\ell},\alpha_k\mathrm{v}})_{\ell \in \mathbb{N}}$
	(see Proposition~\ref{prop:gamma-convergence}),
	for every $k \in \mathbb{N}$
	there exists $\Omega_{k} \in \mathcal{C}^m_{g,\alpha_k\mathrm{v}}(M)$
	such that
	\begin{equation}
		\label{eq:uij-to-wi}
		\lim_{\ell \to \infty}
		\norm{u_{k_\ell}-Z_{\Omega_k}}_{L^1(M,\mathbb{R}^m)}
		= 0,
	\end{equation}
	and
	\begin{equation*}
		\mathcal{P}_g^\omega(\Omega_{k}) 
		\le \liminf_{\ell \to \infty}
		\mathcal{E}_{\varepsilon_{k_\ell},\alpha_k\mathrm{v}}(u_{k_\ell})
		\le \sigma_{\mathrm{v}}(\alpha_k).
	\end{equation*}
	Hence, for every $k\in\mathbb N$, there exists $\ell_k $ such that
	\begin{equation}
		\label{eq:uij-to-wi-2}
		\norm{u_{k_{\ell_k}}-Z_{\Omega_k}}_{L^{1}(M,\mathbb{R}^m)}
		\leq \frac{1}{2}\theta_0\norm{\mathrm{v}}\alpha_k.
	\end{equation}
	Moreover, we have
	\[
		{\mathcal{I}}^\omega_{g}(\alpha_k\mathrm{v})
		\leq{\mathcal{P}}_g^\omega(\Omega_{k})
		\le \liminf_{\ell \to \infty}{\mathcal{E}}_{\varepsilon_{k_\ell},\mathrm{v}_k}(u_{k_\ell})
		\leq {\mathcal{I}}^\omega_{g}(\alpha_k\mathrm{v})+\tau_{\mathrm{v}}(\alpha_k).
	\]
	By~\eqref{eq:tau-zero-infinitesimal-alpha}, the previous chain of inequalities implies that 
	$(\Omega_{k})_{k\in\mathbb N}\subset\mathcal{C}^m_{g,\mathrm{v}}(M)$ is an almost isoperimetric sequence
	(see Definition~\ref{def:almost-isoperimetric-clusters}),
	with $\mathcal{V}_g({\Omega}_{k})=\alpha_k \mathrm{v}$,
	where $\alpha_k$ is infinitesimal.
	As a consequence, we can apply
	Proposition~\ref{prop:concentration-almost-minimizers} to find a sequence of points $(p_k)_{k\in\mathbb N}\subset M$, 
	such that
	\[
		\lim_{k \to \infty}
		\frac{
			\norm{\mathcal{V}_g(\Omega_k \setminus B_g(p_k,\mu_k))}
		}{\alpha_k\norm{\mathrm{v}}}
		= 0.
	\]
	Since $Z_{\Omega_k}$ are combination of 
	indicatrix functions, 
	this implies that for $k$ sufficiently large 
	we have 
	\[
		\norm{Z_{\Omega_k}}_{L^1(M\setminus B_g(p_k,\mu_k),\mathbb{R}^m)}
		\leq \frac{1}{2}\theta_0\norm{\mathrm{v}}\alpha_k,
	\]
	and by using~\eqref{eq:uij-to-wi-2},
	we obtain
	\begin{align*}
		\int_{M \setminus B_g(p_k,\mu_k)} \norm{u_{k_{\ell_k}}} \ud v_{g}
		&\leq
		\norm{u_{k_{\ell_k}} - Z_{\Omega_k}}_{L^1(M\setminus B_g(p_k,\mu_k),\mathbb{R}^m)}
		+
		\norm{Z_{\Omega_k}}_{L^1(M\setminus B_g(p_k,\mu_k),\mathbb{R}^m)}\\
		&\leq
		\norm{u_{k_{\ell_k}} - Z_{\Omega_k}}_{L^1(M,\mathbb{R}^m)}
		+
		\norm{Z_{\Omega_k}}_{L^1(M\setminus B_g(p_k,\mu_k),\mathbb{R}^m)}\\
		&\leq \theta_0\norm{\mathrm{v}}\alpha_k,
	\end{align*}
	which is in contradiction with~\eqref{eq:boh-1}.
\end{proof}
Roughly speaking, the concentration result above allows us
to define a barycenter as in Definition~\ref{def:barycenter},
which is also continuous in the norm topology.
Therefore, we can prove that 
the intrinsic barycenter map is well-defined
and continuous in the norm topology,
thus proving Proposition~\ref{prop:barycenter}.

\begin{proof}[Proof of Proposition~\ref{prop:barycenter}]
	Let us fix a $\theta  \in (0,1/2)$
	and
	let $\alpha_2 = \alpha_2(\theta,\mathrm{v}) > 0$
	and $\varepsilon_2 = \varepsilon_2(\theta,\alpha) > 0$
	be given   
	by Proposition~\ref{prop:almost-min-concentration}.
	As a consequence, 
	if $\alpha \in (0,\alpha_2)$
	and $\varepsilon \in (0,\varepsilon_2)$,
	for any 
	$u \in \mathcal{E}^{\sigma_\mathrm{v}(\alpha)}_{\varepsilon,\alpha\mathrm{v}}$
	there exists $p\in M$
	such that
	\begin{equation}
		\label{concentrationbis}
		\frac{\int_{M\setminus B_g(p,\mu(\alpha\norm{\mathrm{v}})^{1/N})}
		\norm{u} \de v_g}{\alpha\norm{\mathrm{v}}}
		\le \theta,
	\end{equation} 
	which implies 
	\begin{equation*}
		\frac{\int_{B_g(p,\mu(\alpha\norm{\mathrm{v}})^{1/N})}
			\norm{u} \de v_g}{\int_{M}
			\norm{u} \de v_g}=\frac{\int_{M}
			\norm{u} \de v_g-\int_{M\setminus B_g(p,\mu(\alpha\norm{\mathrm{v}})^{1/N})}
		\norm{u} \de v_g}{\norm{\alpha\mathrm{v}}}
		\ge 1-\theta.
	\end{equation*}
	Setting $\eta = 1 - \theta$,
	we have that $\eta \in (1/2,1)$
	and the previous inequality shows that
	$u \in L^1_{r,\eta}(M,\mathbb{R}^m)$,
	with $r = \mu(\alpha\norm{v})^{1/N}$.
	As a consequence, setting $\alpha^*_2(\mathrm{v},r,\eta)
	\in (0,\alpha_2) > 0$
	such that
	\[
		\mu (\alpha \norm{\mathrm{v}})^{1/N} < \frac{r}{2},
		\qquad \forall \alpha \in (0,\alpha^*_2),
	\]
	and setting $\varepsilon^*_2 = \varepsilon_2(\theta,\alpha)>0$,
	the thesis follows.
\end{proof}

\section{Proof of the Main Result}
\label{sec:finalproof}
In this section, we prove that the composition of the photography map
with the barycenter map is homotopic to the identity map,
thus performing \ref{item:s3} and finalizing the proof
of Theorem~\ref{theorem:main}.
\begin{proposition}
	\label{prop:homotopy} 
	For any $r \in (0,r_0)$
	$\eta\in (\frac{1}{2},1)$,
	and $\mathrm{v} \in \mathbb{R}^m_{>0}$,
	there exists $\alpha^*_3 = \alpha^*_3(M,g,r) > 0$
	such that for every $\alpha\in(0,\alpha^*_3)$
	there exists
	$\varepsilon^*_3 = \varepsilon^*_3(\alpha,\mathrm{v},r) > 0$
	such that for every $\varepsilon\in(0,\varepsilon^*_3)$,
	the following inequality holds:
	\[
		\dist_g((\beta_{r,\eta}\circ \varphi_{\varepsilon,\alpha\mathrm{v}})(p),p)
		<r, \qquad \forall p\in M.
	\]
	Moreover, the map
	$\beta_{r,\eta}\circ \varphi_{\varepsilon,\alpha\mathrm{v}}\colon M \to M$
	is homotopic to the identity map.	
\end{proposition}

\begin{proof}
	For any $\alpha \in (0,\alpha^*_1)$,
	where $\alpha^*_1>0$ is given
	by Proposition~\ref{prop:photography},
	let us set 
	\[
		u_{0,\alpha}(p)
		\coloneqq Z_{G_{\mathrm{v},\alpha}(p)},
	\]
	where we recall that $G_{\mathrm{v},\alpha}$
	is defined by~\eqref{eq:def-G}.
	By Remark~\ref{rem:boundenssGvalpha},
	we have that 
	\[
		\supp \norm{u_{0,\alpha}(p)}
		\subset B(p,\alpha^{\frac{1}{N}}R_{g,\mathrm{v}}),
		\qquad \forall p \in M.
	\]
	As a consequence, setting $\alpha_3^* \in (0,\alpha^*_1) $
	such that
	\[
		(\alpha_3^*)^{\frac{1}{N}}R_{g,\mathrm{v}} < \frac{r}{2},
	\]
	we have that for any $\alpha \in (0,\alpha_3^*)$
	the following inequality holds:
	\[
		\dist_g(\beta_{r,\eta}(u_{0,\alpha}(p)),p)
		< \frac{r}{2},
		\qquad \forall p \in M,
	\]
	where $\eta \in (1/2,1)$ can be arbitrary.
	Moreover, by the construction of the photography map
	as an element of the recovery sequence of 
	$u_{0,\alpha}(p)$,
	we have that
	$\varphi_{\varepsilon,\alpha\mathrm{v}}(p)$
	converges in $L^1$ norm to
	$u_{0,\alpha}(p)$.
	By the continuity of $\beta_{r,\eta}$,
	this implies that there exists
	$\varepsilon_3^* > 0 $
	such that, for any $\varepsilon \in (0,\varepsilon_3^*)$,
	we have 
	\[
		\dist_g\big(\beta_{r,\eta}(\varphi_{\varepsilon,\alpha\mathrm{v}}(p)),
			\beta_{r,\eta}(u_{0,\alpha}(p))
		\big)
		< \frac{r}{2},
	\]
	As a consequence, we obtain that 
	for any $\alpha \in (0,\alpha_3^*)$
	and any $\varepsilon \in (0,\varepsilon_3^*)$
	we have
	\[
		\dist_g\big((\beta_{r,\eta}\circ      \varphi_{\varepsilon,\alpha\mathrm{v}})(p),p\big)
		< 
		\dist_g\big(\beta_{r,\eta}(\varphi_{\varepsilon,\alpha\mathrm{v}}(p)),
			\beta_{r,\eta}(u_{0,\alpha}(p))
		\big)
		+
		\dist_g\big(\beta_{r,\eta}(u_{0,\alpha}(p)),p\big),
		\quad \forall p \in M.
	\]

	To finish the proof, we define the homotopy 
	$F:[0,1] \times M \rightarrow M$ by
	\[
		F(t, p)\coloneqq \exp_{p}(t \exp _{p}^{-1}(\beta_{r,\eta}\circ\varphi_{\varepsilon,\alpha\mathrm{v}}(p))),
	\]
	which by definition satisfies $F(0, p)=p$ and
	$F(1, p)=(\beta_{r,\eta}\circ\varphi_{\varepsilon,\mathrm{v}})(p)$ for every $p\in M$. 
	Also, the continuity of $p\mapsto F(\cdot,p)$
	follows from the standard properties of the exponential map.
\end{proof}

As a consequence of 
Proposition~\ref{prop:photography},
Proposition~\ref{prop:barycenter}
and
Proposition~\ref{prop:homotopy},
all the hypotheses of the 
photography method, {\it i.e.}, Theorem~\ref{theorem:abstract-photography},
are satisfied in our setting,
provided that $\alpha$ and $\varepsilon$ are sufficiently small.
Therefore, Theorem~\ref{theorem:main} directly follows
from \cite[Theorem C]{MR4701348}, as we show 
in this last proof.

\begin{proof}[Proof of Theorem~\ref{theorem:main}]
	Let us fix $r \in \left(0,\frac{r_0}{2}\right)$
	and $\eta \in (1/2,1)$.
	Let
	$\alpha^*=\min\{\alpha^*_1,\alpha^*_2,\alpha^*_3\}$
	and, for any $\alpha \in (0,\alpha^*)$,
	let
	$\varepsilon^*  =\min\{\varepsilon^*_1,\varepsilon^*_2,\varepsilon^*_3\}$,
	so that, 
	for any $\varepsilon \in (0,\varepsilon^*)$,
	Proposition~\ref{prop:photography},
	Proposition~\ref{prop:barycenter}
	and Proposition~\ref{prop:homotopy} hold.

	Let us apply Theorem~\ref{theorem:abstract-photography},
	where $X=(M,\dist_g)$,
	$\mathcal{E}={\mathcal{E}}_{\varepsilon,\alpha\mathrm{v}}$,
	$\mathfrak{M}=W^{1,2}_\mathrm{v}(M,\mathbb{R}^m)$,
	$\Psi_{L}=\beta_{r,\eta}$ 
	and $\Psi_{R}=\varphi_{\varepsilon,\alpha\mathrm{v}}$.

	Since $\mathcal{E}_{\varepsilon,\alpha\mathrm{v}}$
	is trivially lower-bounded and satisfies the Palais-Smale condition,
	as demonstrated in \cite[Lemma 5.5]{MR4701348},
	the conditions \ref{itm:E1} and \ref{itm:E2} are trivially satisfied.
	Whence, 
	by Proposition~\ref{prop:photography}
	we have that
	$ \varphi_{\varepsilon,\mathrm{v}}\in C^0(M,\mathcal{E}_{\varepsilon,\alpha\mathrm{v}}^{\sigma_{\mathrm{v}}(\alpha)})$,
	and by Proposition~\ref{prop:barycenter}
	$\beta_{r,\eta}\in C^0(\mathcal{E}_{\varepsilon,\alpha\mathrm{v}}^{\sigma_{\mathrm{v}}(\alpha)},M)$.
	Moreover, 
	by Proposition~\ref{prop:homotopy}
	$\varphi_{\varepsilon,\mathrm{v}}\circ \beta_{r,\eta}$
	is homotopic to the identity map.
	Thus, the condition \ref{itm:E3} also holds and
	Theorem~\ref{theorem:main} follows from
	Theorem~\ref{theorem:abstract-photography}:
	if all the critical points of $\mathcal{E}_{\varepsilon,\alpha\mathrm{v}}$
	are nondegenerate, then (ii) holds;
	otherwise, we have only the estimation provided by (i).
\end{proof}
\color{black}

\section*{List of Notations}
\label{sec:notations}

We report the notations we used frequently throughout the text for easy reference.   

\begin{itemize}
	\item[] $h=\mathrm{o}(f)$ as $x\rightarrow x_0$ for $x_0\in\mathbb{R}\cup\{\pm\infty\}$, if $\lim_{x\rightarrow x_0}(h/f)(x)=0$;
	\item[] $\mathrm{o}_k(1)$ means a quantity that goes to zero as $k\to\infty$;
	\item[] $\mathbb{R}^m_{>0}$ ($\mathbb{R}^m_{+}$) is the space of vectors with positive (nonnegative) components;
	\item[] $(M,g)$ is a closed Riemannian manifold with injectivity radius $\inj_M>0$;
	\item[] $N\ge 2$ is the dimension of $M$;
	\item[] $(\mathbb{R}^N,\delta)$ is the Euclidean space with its canonical flat metric;
	\item[] $B_g(p,r)$ is the geodesic ball with $p\in M$ and $r>0$; we denote $B_{\delta}(p,r)=B(p,r)$;
	\item[] $\Delta_g$ is the Laplace--Beltrami operator;
	\item[] $V_g$ is the vectorial volume function;
	\item[] $TM$ is the tangent bundle with sections $\Gamma(TM)$;
	\item[] $m\ge 2$ is the number of equations;
	\item[] $W$ is a multi-well potential;
	\item[] $\lambda\in \mathbb R$ is an $m$--dimensional Lagrange multiplier;
	\item[] $0<\varepsilon\ll1$ is the temperature parameter;
	\item[] $0<\alpha\ll1$ is the order rescaling parameter;
	\item[] $\mathrm{v}\in \mathbb{R}^m_{>0}$ is a vectorial volume constraint;
	\item[] $\mathcal{C}^{\ell,\gamma}(M,\mathbb R^m)$ is the H\"{o}lder space of $m$-maps with $\ell\in\mathbb N_0$ and $\gamma\in (0,1)$;
	\item[] $W^{\ell,q}(M,\mathbb R^m)$ is the Sobolev space of $m$-maps, with $\ell\in\mathbb N$ and $q\in[1,\infty]$;
	\item[] $L^{q}(M,\mathbb R^m)$ is the Lebesgue space of $m$-maps with $q\in[1,\infty]$;
	\item[] $L^1_{r,\eta}(M,\mathbb R^m)$ is the space of $(r,\eta)$-concentrated of $m$-maps;
	\item[] $W^{1,2}_\mathrm{v}(M,\mathbb{R}^m)$ is the space of functions with vectorial volume $\mathrm{v}$;
	\item[] $\mathcal{E}_{\varepsilon,\mathrm{v}}$ is the Allen-Cahn energy functional restricted to $W^{1,2}_\mathrm{v}(M)$;
	\item[] $\mathcal{E}^c_{\varepsilon,\mathrm{v}}$ the sublevels of the Allen-Cahn energy functional;
	\item[] $\mathcal{Z}=\{\mathbf{z}_0,\dots,\mathbf{z}_m\}$ are the zeros of $W$;
	\item[] $\omega_{ij}$ are the weights of the multi-isoperimetric problem;
	\item[] $\cat(M)$ is the Lusternick--Schnirelmann category of $M$;
	\item[] $\mathcal{P}_1(M)$ is the Poincar\'e polynomial of $M$ evaluated at $t=1$;

	\item[] $2^{\#}\coloneqq\frac{2N-1}{N-1}$ is the (lower) critical Sobolev exponent;
	\item[] $2^{*}\coloneqq\frac{2N}{N-2}$ is the (upper) critical Sobolev exponent;
	\item[] $\mathcal{C}^m_g(M)$ is the set of $m$-clusters on $M$;
	\item[] $\mathcal{V}_g$ is the vectorial volume function acting on $m$-clusters;
	\item[] $J$ is the map that associates vectorial volumes with almost
		isoperimetric clusters;
	\item[] $\Pi$ is the projection map of the $m$-cluster near a point;
	\item[]$R_\varepsilon$ is the map that associates a recovery sequence to each point and volume;
	\item[] $\varphi_{\varepsilon,\mathrm{v}}$ is the photography map;
	\item[] $\beta_{r,\eta}$ is the intrinsic barycenter map restricted to $L^1_{r,\eta}(M,\mathbb R^m)$;
	\item[] $\mathcal{C}^m_{g,\mathrm{v}}(M)$ is the set of $m$-clusters in $M$ with volume $\mathrm{v}$;
	\item[] $\mathcal{P}^{\omega}_g$ is the $\omega$-weighted multi-perimeter functional;
	\item[] $\mathcal{I}^\omega_{g}$ is the $\omega$-weighted multi-isoperimetric function.
\end{itemize}


\begin{thebibliography}{10}

	\bibitem{MR4761862}
	S.~Alarc\'{o}n, J.~Petean and C.~Rey, Multiplicity results for constant
	{Q}-curvature conformal metrics, \emph{Calc. Var. Partial Differential
	Equations} {\bf 63} (2024) Paper No. 146.

	\bibitem{allen-cahn}
	S.~M. Allen and J.~W. Cahn, A microscopic theory for antiphase boundary motion
	and its application to antiphase domain coarsening, \emph{Acta Metallurgica}
	{\bf 27} (1979) 1085--1095.

	\bibitem{MR420406}
	F.~Almgren, Existence and regularity almost everywhere of solutions to elliptic
	variational problems with constraints, \emph{Mem. Amer. Math. Soc.} {\bf 4}
	(1976) viii+199.

	\bibitem{MR1070482}
	L.~Ambrosio and A.~Braides, Functionals defined on partitions in sets of finite
	perimeter. {II}. {S}emicontinuity, relaxation and homogenization, \emph{J.
	Math. Pures Appl. (9)} {\bf 69} (1990) 307--333.

	\bibitem{MR4701348}
	J.~H. Andrade, J.~Conrado, S.~Nardulli, P.~Piccione and R.~Resende,
	Multiplicity of solutions to the multiphasic {A}llen-{C}ahn-{H}illiard system
	with a small volume constraint on closed parallelizable manifolds, \emph{J.
	Funct. Anal.} {\bf 286} (2024) Paper No. 110345, 61.

	\bibitem{arXiv:2407.06934}
	J.~H. Andrade, T.~König, J.~Ratzkin and J.~Wei, Quantitative stability of the
	total $Q$-curvature near minimizing metrics (2024), arXiv:2407.06934
	[math.AP].

	\bibitem{MR4467099}
	G.~Antonelli, S.~Nardulli and M.~Pozzetta, The isoperimetric problem {\it via}
	direct method in noncompact metric measure spaces with lower {R}icci bounds,
	\emph{ESAIM Control Optim. Calc. Var.} {\bf 28} (2022) Paper No. 57, 32.

	\bibitem{MR4745750}
	G.~Antonelli, E.~Pasqualetto, M.~Pozzetta and D.~Semola, Asymptotic
	isoperimetry on non collapsed spaces with lower {R}icci bounds, \emph{Math.
	Ann.} {\bf 389} (2024) 1677--1730.

	\bibitem{MR1051228}
	S.~Baldo, Minimal interface criterion for phase transitions in mixtures of
	{C}ahn-{H}illiard fluids, \emph{Ann. Inst. H. Poincar\'{e} Anal. Non
	Lin\'{e}aire} {\bf 7} (1990) 67--90.

	\bibitem{MR1322324}
	V.~Benci, Introduction to {M}orse theory: a new approach, \emph{Topological
	nonlinear analysis}, \emph{Progr. Nonlinear Differential Equations Appl.},
	vol.~15, Birkh\"{a}user Boston, Boston, MA (1995) 37--177.

	\bibitem{MR1088278}
	V.~Benci and G.~Cerami, The effect of the domain topology on the number of
	positive solutions of nonlinear elliptic problems, \emph{Arch. Rational Mech.
	Anal.} {\bf 114} (1991) 79--93.

	\bibitem{MR1384393}
	V.~Benci and G.~Cerami, Multiple positive solutions of some elliptic problems
	via the {M}orse theory and the domain topology, \emph{Calc. Var. Partial
	Differential Equations} {\bf 2} (1994) 29--48.

	\bibitem{MR4644903}
	V.~Benci, D.~Corona, S.~Nardulli, L.~E. Osorio~Acevedo and P.~Piccione,
	Corrigendum to: ``{L}usternik-{S}chnirelman and {M}orse theory for the van
	der {W}aals-{C}ahn-{H}illiard equation with volume constraint'' [{N}onlinear
	{A}nal. 220 (2022) 112851], \emph{Nonlinear Anal.} {\bf 238} (2024) Paper No.
	113389, 9.

	\bibitem{MR4396580}
	V.~Benci, S.~Nardulli, L.~E.~O. Acevedo and P.~Piccione,
	Lusternik-{S}chnirelman and {M}orse theory for the van der
	{W}aals--{C}ahn--{H}illiard equation with volume constraint, \emph{Nonlinear
	Anal.} {\bf 220} (2022) Paper No. 112851, 29.

	\bibitem{MR4073210}
	V.~Benci, S.~Nardulli and P.~Piccione, Multiple solutions for the van der
	{W}aals--{A}llen--{C}ahn--{H}illiard equation with a volume constraint,
	\emph{Calc. Var. Partial Differential Equations} {\bf 59} (2020) Paper No.
	64, 30.

	\bibitem{Berger}
	M. S. Berger, \emph{Nonlinearity and functional analysis},
	Pure and Applied Mathematics, Academic Press, New York-London, 1977.

	\bibitem{cahn-hilliard}
	J.~W. Cahn and J.~E. Hilliard, Free energy of a nonuniform system, I.
	Interfacial free energy, \emph{J. Chem. Phys} {\bf 28} (1958) 258--267.

	\bibitem{arXiv:2306.07100}
	M.~Caselli, E.~Florit-Simon and J.~Serra, Yau's conjecture for nonlocal minimal
	surfaces (2024), arXiv:2306.07100 [math.DG].

	\bibitem{MR1646619}
	S.~Cingolani and M.~Lazzo, Multiple semiclassical standing waves for a class of
	nonlinear {S}chr\"{o}dinger equations, \emph{Topol. Methods Nonlinear Anal.}
	{\bf 10} (1997) 1--13.

	\bibitem{MR1734531}
	S.~Cingolani and M.~Lazzo, Multiple positive solutions to nonlinear
	{S}chr\"{o}dinger equations with competing potential functions, \emph{J.
	Differential Equations} {\bf 160} (2000) 118--138.

	\bibitem{arXiv:2401.17847}
	D.~Corona, S.~Nardulli, R.~Oliver-Bonafoux, G.~Orlandi and P.~Piccione,
	Multiplicity results for mass constrained Allen-Cahn equations on Riemannian
	manifolds with boundary (2024), arXiv:2401.17847 [math.AP].

	\bibitem{MR3748585}
	G.~Dal~Maso, I.~Fonseca and G.~Leoni, Asymptotic analysis of second order
	nonlocal {C}ahn-{H}illiard-type functionals, \emph{Trans. Amer. Math. Soc.}
	{\bf 370} (2018) 2785--2823.

	\bibitem{MR4314216}
	G.~de~Paula~Ramos, Nondegenerate solutions for constrained semilinear elliptic
	problems on {R}iemannian manifolds, \emph{NoDEA Nonlinear Differential
	Equations Appl.} {\bf 28} (2021) Paper No. 64, 11.

	\bibitem{MR4498838}
	A.~Dey, A comparison of the {A}lmgren-{P}itts and the {A}llen-{C}ahn min-max
	theory, \emph{Geom. Funct. Anal.} {\bf 32} (2022) 980--1040.

	\bibitem{MR4649390}
	A.~Dey, Existence of multiple closed {CMC} hypersurfaces with small mean
	curvature, \emph{J. Differential Geom.} {\bf 125} (2023) 379--403.

	\bibitem{MR4459029}
	G.~Di~Matteo and A.~Malchiodi, Double bubbles with high constant mean
	curvatures in {R}iemannian manifolds, \emph{Nonlinear Anal.} {\bf 224} (2022)
	Paper No. 113088, 41.

	\bibitem{MR4427104}
	M.~Engelstein, R.~Neumayer and L.~Spolaor, Quantitative stability for
	minimizing {Y}amabe metrics, \emph{Trans. Amer. Math. Soc. Ser. B} {\bf 9}
	(2022) 395--414.

	\bibitem{MR3945835}
	P.~Gaspar and M.~A.~M. Guaraco, The {W}eyl law for the phase transition
	spectrum and density of limit interfaces, \emph{Geom. Funct. Anal.} {\bf 29}
	(2019) 382--410.

	\bibitem{MR1453735}
	G.~Giacomin and J.~L. Lebowitz, Phase segregation dynamics in particle systems
	with long range interactions. {I}. {M}acroscopic limits, \emph{J. Statist.
	Phys.} {\bf 87} (1997) 37--61.

	\bibitem{MR1638739}
	G.~Giacomin and J.~L. Lebowitz, Phase segregation dynamics in particle systems
	with long range interactions. {II}. {I}nterface motion, \emph{SIAM J. Appl.
	Math.} {\bf 58} (1998) 1707--1729.

	\bibitem{gibbs}
	J.~W. Gibbs, On the equilibrium of heterogeneous substances, \emph{Trans. Conn.
	Acad. Arts Sci} {\bf 3} (1874–1878) 108--248.

	\bibitem{MR3743704}
	M.~A.~M. Guaraco, Min-max for phase transitions and the existence of embedded
	minimal hypersurfaces, \emph{J. Differential Geom.} {\bf 108} (2018) 91--133.

	\bibitem{MR855305}
	M.~E. Gurtin, On phase transitions with bulk, interfacial, and boundary energy,
	\emph{Arch. Rational Mech. Anal.} {\bf 96} (1986) 243--264.

	\bibitem{MR2160744}
	D.~Henry, \emph{Perturbation of the boundary in boundary-value problems of
		partial differential equations}, \emph{London Mathematical Society Lecture
	Note Series}, vol. 318, Cambridge University Press, Cambridge (2005), with
	editorial assistance from Jack Hale and Ant\^{o}nio Luiz Pereira.

	\bibitem{MR1803974}
	J.~E. Hutchinson and Y.~Tonegawa, Convergence of phase interfaces in the van
	der {W}aals-{C}ahn-{H}illiard theory, \emph{Calc. Var. Partial Differential
	Equations} {\bf 10} (2000) 49--84.

	\bibitem{MR442975}
	H.~Karcher, Riemannian center of mass and mollifier smoothing, \emph{Comm. Pure
	Appl. Math.} {\bf 30} (1977) 509--541.

	\bibitem{MR3847750}
	T.~Laux and T.~M. Simon, Convergence of the {A}llen-{C}ahn equation to
	multiphase mean curvature flow, \emph{Comm. Pure Appl. Math.} {\bf 71} (2018)
	1597--1647.

	\bibitem{MR3145921}
	G.~R. Lawlor, Double bubbles for immiscible fluids in {$\mathbf{R}^n$}, \emph{J.
	Geom. Anal.} {\bf 24} (2014) 190--204.

	\bibitem{MR888880}
	J.~M. Lee and T.~H. Parker, The {Y}amabe problem, \emph{Bull. Amer. Math. Soc.
	(N.S.)} {\bf 17} (1987) 37--91.

	\bibitem{leonardi2001}
	G.~P. Leonardi, Infiltrations in immiscible fluids systems, \emph{Proceedings
	of the Royal Society of Edinburgh Section A: Mathematics} {\bf 131} (2001)
	425--436.

	\bibitem{MR2976521}
	F.~Maggi, \emph{Sets of finite perimeter and geometric variational problems},
	\emph{Cambridge Studies in Advanced Mathematics}, vol. 135, Cambridge
	University Press, Cambridge (2012), an introduction to geometric measure
	theory.

	\bibitem{MR3674223}
	F.~C. Marques and A.~Neves, Existence of infinitely many minimal hypersurfaces
	in positive {R}icci curvature, \emph{Invent. Math.} {\bf 209} (2017)
	577--616.

	\bibitem{MR3953507}
	F.~C. Marques, A.~Neves and A.~Song, Equidistribution of minimal hypersurfaces
	for generic metrics, \emph{Invent. Math.} {\bf 216} (2019) 421--443.

	\bibitem{MR2560131}
	A.~M. Micheletti and A.~Pistoia, Generic properties of singularly perturbed
	nonlinear elliptic problems on {R}iemannian manifold, \emph{Adv. Nonlinear
	Stud.} {\bf 9} (2009) 803--813.

	\bibitem{MR0440554}
	J.~W. Milnor and J.~D. Stasheff, \emph{Characteristic classes}, Annals of
	Mathematics Studies, No. 76, Princeton University Press, Princeton, N. J.;
	University of Tokyo Press, Tokyo (1974).

	\bibitem{MR866718}
	L.~Modica, The gradient theory of phase transitions and the minimal interface
	criterion, \emph{Arch. Rational Mech. Anal.} {\bf 98} (1987) 123--142.

	\bibitem{MR473971}
	L.~Modica and S.~Mortola, Il limite nella {$\Gamma$}-convergenza di una
	famiglia di funzionali ellittici, \emph{Boll. Un. Mat. Ital. A (5)} {\bf 14}
	(1977) 526--529.

	\bibitem{MR3497381} F.~Morgan, {\it Geometric measure theory}, fifth edition, Elsevier/Academic Press, Amsterdam, 2016;

	\bibitem{MR2529468}
	S.~Nardulli, The isoperimetric profile of a smooth {R}iemannian manifold for
	small volumes, \emph{Ann. Global Anal. Geom.} {\bf 36} (2009) 111--131.

	\bibitem{MR4130849}
	S.~Nardulli and L.~E. Osorio, Sharp isoperimetric inequalities for small
	volumes in complete noncompact {R}iemannian manifolds of bounded geometry
	involving the scalar curvature, \emph{Int. Math. Res. Not. IMRN} {\bf 2020}
	(2020) 4667--4720.

	\bibitem{MR2032110}
	F.~Pacard and M.~Ritor\'{e}, From constant mean curvature hypersurfaces to the
	gradient theory of phase transitions, \emph{J. Differential Geom.} {\bf 64}
	(2003) 359--423.

	\bibitem{MR3912791}
	J.~Petean, Multiplicity results for the {Y}amabe equation by
	{L}usternik-{S}chnirelmann theory, \emph{J. Funct. Anal.} {\bf 276} (2019)
	1788--1805.

	\bibitem{MR1480173} P. Petersen, {\it Riemannian geometry}, Graduate Texts in Mathematics, 171, Springer, New York, 1998.

	\bibitem{MR2769110}
	A.~Pisante and M.~Ponsiglione, Phase transitions and minimal hypersurfaces in
	hyperbolic space, \emph{Comm. Partial Differential Equations} {\bf 36} (2011)
	819--849.

	\bibitem{MR1983191}
	X.~Ren and J.~Wei, Triblock copolymer theory: free energy, disordered phase and
	weak segregation, \emph{Phys. D} {\bf 178} (2003) 103--117.

	\bibitem{MR3302114}
	X.~Ren and J.~Wei, A double bubble assembly as a new phase of a ternary
	inhibitory system, \emph{Arch. Ration. Mech. Anal.} {\bf 215} (2015)
	967--1034.

	\bibitem{MR4588150}
	R. Resende~de~Oliveira, On clusters and the multi isoperimetric
	profile in Riemannian manifolds with bounded geometry, \emph{J. Dyn. Control Syst.} {\bf 29} (2023),
	no.~2, 419--441; 

	\bibitem{MR4564260}
	A.~Song, Existence of infinitely many minimal hypersurfaces in closed
	manifolds, \emph{Ann. of Math. (2)} {\bf 197} (2023) 859--895.

	\bibitem{MR930124}
	P.~Sternberg, The effect of a singular perturbation on nonconvex variational
	problems, \emph{Arch. Rational Mech. Anal.} {\bf 101} (1988) 209--260.

	\bibitem{MR2948876}
	Y.~Tonegawa and N.~Wickramasekera, Stable phase interfaces in the van der
	{W}aals--{C}ahn--{H}illiard theory, \emph{J. Reine Angew. Math.} {\bf 668}
	(2012) 191--210.

	\bibitem{vanderwaals}
	J.~D. van~der Waals, The thermodynamic theory of capillarity under the
	hypothesis of a continuous variation of density, \emph{J. Stat.
	Phys.} {\bf 20} (1979) 200--244.

	\bibitem{MR1402391}
	B.~White, Existence of least-energy configurations of immiscible fluids,
	\emph{J. Geom. Anal.} {\bf 6} (1996) 151--161.

	\bibitem{MR645762}
	S.~T. Yau, Problem section, \emph{Seminar on {D}ifferential {G}eometry},
	\emph{Ann. of Math. Stud.}, vol. 102, Princeton Univ. Press, Princeton, N.J.
	(1982) 669--706.

	\bibitem{MR4011704}
	X.~Zhou and J.~J. Zhu, Min-max theory for constant mean curvature
	hypersurfaces, \emph{Invent. Math.} {\bf 218} (2019) 441--490.

\end{thebibliography}
\end{document}